\documentclass[11]{article}

\usepackage[english]{babel}
\usepackage{comment}
\usepackage[letterpaper,top=2cm,bottom=2cm,left=3cm,right=3cm,marginparwidth=1.75cm]{geometry}

\usepackage{cite}
\usepackage{subfigure}
\usepackage{algorithm}
\usepackage[noend]{algpseudocode}
\usepackage{amsmath}
\usepackage{graphicx}
\usepackage[utf8]{inputenc} 
\usepackage{amsfonts}
\usepackage{amsthm}
\usepackage{mathtools}
\usepackage{amssymb}
\usepackage{multirow}
\usepackage[colorlinks=true, allcolors=blue]{hyperref}
\newtheorem{theorem}{Theorem}[section]
\newtheorem{corollary}{Corollary}[theorem]
\newtheorem{lemma}{Lemma}[section]
\newtheorem{prop}{Proposition}[section]
\newtheorem{remark}{Remark}[section]
\newtheorem{definition}{Definition}[section]
\title{Randomized strong rank-revealing QR for column subset selection and low-rank matrix approximation}
\author{Laura Grigori\thanks{PSI, Center for Scientific Computing, Theory and Data, Villigen, and EPFL, Institute of Mathematics, Lausanne, Switzerland}, Zhipeng Xue\thanks{EPFL, Institute of Mathematics, Lausanne, Switzerland}}
\date{}
\makeatletter
\newcommand\tabcaption{\def\@captype{table}\caption}
\newcommand\figcaption{\def\@captype{figure}\caption}
\makeatother
\algdef{SE}[SUBALG]{Indent}{EndIndent}{}{\algorithmicend\ }%
\algtext*{Indent}
\algtext*{EndIndent}
\newcommand{\lab}{\left|}
\newcommand{\rab}{\right|}
\newcommand{\be}{\begin{eqnarray}}
	\newcommand{\ee}{\end{eqnarray}}
\newcommand{\beno}{\begin{eqnarray*}}
	\newcommand{\eeno}{\end{eqnarray*}}
\newcommand{\barr}[1]{\begin{array}{#1}}
	\newcommand{\earr}{\end{array}}

\newcommand{\R}{\mathbb{R}}
\newcommand{\Rmn}{\mathbb{R}^{m\times n}}

\newcommand{\Rnn}{\mathbb{R}^{n\times n}}
\newcommand{\Rdm}{\mathbb{R}^{d\times m}}

\newcommand{\Rdn}{\mathbb{R}^{d\times n}}

\newcommand{\md}{\delta}
\usepackage{tikz}
\newcommand\parallelogram{%
    \mathord{\text{%
        \tikz[baseline] \draw (0,.1ex) -- (.8em,.1ex) -- (1em,1.6ex) -- (.2em,1.6ex) -- cycle;}}}
\newcommand{\me}{\epsilon}

\newcommand{\ms}{\sigma}

\newcommand{\mr}{\gamma}
\newcommand{\mw}{\omega}

\algrenewcommand\algorithmicrequire{\textbf{Input:}}
\algrenewcommand\algorithmicensure{\textbf{Output:}}

\newcommand{\cK}{{\mathcal K}}
\newcommand{\cU}{{\mathcal U}}
\newcommand{\cV}{{\mathcal V}}

\newcommand{\cP}{{\mathcal P}}
\newcommand{\cR}{{\mathcal R}}
\newcommand{\bA}{{\textbf{A}}}
\newcommand{\bB}{{\textbf{B}}}

\newcommand{\bD}{{\textbf{D}}}
\newcommand{\bL}{{\textbf{L}}}

\newcommand{\bP}{{\textbf{P}}}
\newcommand{\bM}{{\textbf{M}}}
\newcommand{\bH}{{\textbf{H}}}
\newcommand{\bK}{{\textbf{K}}}
\newcommand{\bO}{{\textbf{0}}}
\newcommand{\bQ}{{\textbf{Q}}}

\newcommand{\bI}{{\textbf{I}}}
\newcommand{\bR}{{\textbf{R}}}

\newcommand{\bb}{{\textbf{b}}}

\newcommand{\bu}{{\textbf{u}}}
\newcommand{\bU}{{\textbf{U}}}
\newcommand{\bV}{{\textbf{V}}}
\newcommand{\bY}{{\textbf{Y}}}
\newcommand{\bq}{{\textbf{q}}}
\newcommand{\bx}{{\textbf{x}}} 

\newcommand{\bv}{{\textbf{v}}}
\newcommand{\bW}{{\textbf{W}}}

\newcommand{\mmr}{{\pmb{\gamma}}}
\newcommand{\mmP}{{\pmb{\Pi}}}
\newcommand{\mmO}{{\pmb{\Omega}}}
\newcommand{\by}{{\textbf{y}}}
\newcommand{\ijk}{1\leq i \leq k,\ 1\leq j\leq n-k}

\newcommand{\ang}{\measuredangle}
\newcommand{\ziji}{\subset}
\DeclareMathOperator*{\argmax}{arg\,max}
\DeclareMathOperator*{\argmin}{arg\,min}

\begin{document}
\maketitle	
\begin{abstract}
We discuss a randomized strong rank-revealing QR factorization that effectively reveals the spectrum of a matrix $\bM$. 
%It relies on the QR factorization with column interchanges. 
This factorization can be used to address problems such as selecting a subset of the columns of $\bM$, computing its low-rank approximation, estimating its rank, or approximating its null space. Given a random sketching matrix $\mmO$ that satisfies the $\me$-embedding property for a subspace within the range of $\bM$, the factorization relies on selecting columns that allow us to reveal the spectrum via a deterministic strong rank-revealing QR factorization of $\bM^{sk} = \mmO\bM$, the sketch of $\bM$. We show that this selection leads to a factorization with strong rank-revealing properties, making it suitable for approximating the singular values of $\bM$.
\end{abstract}

\section{Introduction}
Revealing the spectrum of a matrix using the QR factorization with column interchanges is essential in many computations, including column subset selection, low-rank matrix approximation, rank estimation, and null space approximation.
%such as selecting a subset of its columns, computing its low-rank approximation, estimating its rank, or approximating its null space. 
Given a matrix $\bM \in \mathbb{R}^{m \times n}$, this factorization is defined as
\be \label{def_partqr} \bM \mmP = \bQ \bR =
\bQ \begin{pmatrix} \bR_{11} & \bR_{12}\\ \bO & \bR_{22}
\end{pmatrix},
\ee where $\bQ \in \R^{m\times m}$ is an orthogonal matrix, $\bR_{11} \in \R^{k\times k}$ is an upper-triangular matrix, and $\mmP \in \R^{n\times n }$ is a permutation matrix chosen to reveal the spectrum of $\bM$ \cite{rrqr_survey_ipsen, srrqr}.  Such a factorization is called strong rank-revealing if it satisfies the following properties:
\begin{equation}
\begin{split}
  \label{def_srrqr}
  1  \leq  \frac{\ms_i(\bM)}{\ms_i(\bR_{11})} \leq q(k,n) \;  \ \text{and }&\ \;1  \leq \frac{\ms_j(\bR_{22})}{\ms_{j+k}(\bM)} \leq q(k,n), \\
  |(\bR_{11}^{-1} \bR_{12})_{i,j}|& \leq f,
\end{split}
\end{equation}
for any $1\leq i \leq k$ and $1 \leq j \leq n-k$, where $f$ is a small
constant, $q(k,n)$ is a function bounded by a low-degree polynomial in $k$ and $n$, and assuming that $\ms_{j+k}(\bM) \neq 0$.  The first inequalities ensure that the singular values of $\bR_{11}$ approximate well the $k$ leading singular values of $\bM$ and those of $\bR_{22}$ approximate well the remaining $\min(m,n) - k$ singular values of $\bM$.  The lower bound in both inequalities is satisfied for any permutation by the interlacing property of singular values\cite{golub2013matrix}.   The last inequality allows us to bound the elements of $\bR_{11}^{-1} \bR_{12}$. Thus, when $\ms_{max}(\bR_{22})$ is small, $\mmP \begin{pmatrix} -\bR_{11}^{-1}\bR_{12}\\ \bI_{n-k} \end{pmatrix}$ can be used to approximate the null space of $\bM$. 

One of the most used algorithms for solving this problem is the QR
factorization with column pivoting (QRCP) \cite{GOLUB1965}.  At each step of the factorization, the column of maximum norm is permuted to the leading position, and then a Householder reflector is applied to update the factorization. 
\begin{comment}
After $k$ steps, the resulting factorization
satisfies the following properties:
\be
\begin{split}
     & |r_{11}| \geq |r_{22}| \geq \cdots \geq |r_{kk}|, \\
      &(r_{kk})^2 \geq \sum\limits_{i=k+1}^{m}(r_{i,j})^2,\ \forall\  j\geq k+1.
\end{split}
\ee    
\end{comment}
Although it performs well in practice, except in rare cases such as the Kahan matrix \cite{Kahan_book}, the bounds on the singular values for QRCP grow exponentially with the rank $k$, 
 $
 %\label{qrcp_rrqr_bound}
     1  \leq  \frac{\ms_i(\bM)}{\ms_i(\bR_{11})} \leq 2^i \sqrt{n-i}\; , \;1  \leq \frac{\ms_j(\bR_{22})}{\ms_{j+k}(\bM)} \leq 2^k \sqrt{n-k}.$ 
The strong rank-revealing QR factorization introduced in \cite{srrqr} satisfies \eqref{def_srrqr} with 
\begin{equation}
q(k,n) = \sqrt{1 + f^2 k (n-k)}.    
\end{equation}
For a given $k$, it is obtained by first performing QR with column pivoting, followed by additional permutations until the strong rank-revealing property in \eqref{def_srrqr} holds. It can also be obtained by progressively increasing the value of $k$ and identifying for each value of $k$ the columns that need to be interchanged such that \eqref{def_srrqr} holds. We refer to this factorization as strong RRQR or SRRQR.

Since these spectrum-revealing factorizations involve permuting columns at each step of the factorization, they do not fully exploit BLAS3 optimized kernels and they involve significant data movement. As a result, their performance is significantly lower than that of the QR factorization without pivoting. In the recent years this problem has been addressed by using randomization, a technique that allows embedding large dimensional subspaces into low dimensional ones while preserving some geometry, as inner products of vectors or their norms. It was successfully used for solving least squares problems, computing low-rank matrix approximations, or solving linear systems of equations and eigenvalue problems, for more details see e.g. \cite{Martinsson_Tropp_2020, murray2023randomizednumericallinearalgebra}. For a set of vectors belonging to a subspace $\cV \subset \R^m$, random sketching relies on a linear map $\mmO \in \mathbb{R}^{d \times m}$, with $d \ll m$, that satisfies the following property, known as the $\me$-embedding property,  
\be
    |\langle \mmO \bx, \mmO \by \rangle - \langle \bx, \by \rangle| \leq \me \|\bx\|_2 \|\by\|_2, \quad \forall \bx, \by \in \cV \subset \R^m,
\ee  
where the constant $\me$ represents the distortion parameter of the map $\mmO$.  Randomized QR with column pivoting factorizations were introduced in \cite{ming_rqrcp, hqrrp}, where the column pivots are selected from the sketch of $\bM$, $\bM^{sk} = \mmO \bM$, through its QR factorization with column pivoting (QRCP).  Thus $k$ pivot columns can be chosen at once in each iteration, and since the sketch dimension depends on $k$, it can be much smaller than the dimensions of the matrix. As a consequence, the computational and communication costs associated with pivoting is significantly reduced and the QR factorization of the matrix $\bM$ of large dimensions is done without pivoting. In addition, the factorization of the matrix $\bM$ can use level-3 optimized BLAS routines. However, since the factorization of $\bM^{sk}$ relies on QRCP, it inherits the exponential growth of the bounds on singular values with increasing $k$. Indeed, it is shown in \cite{rqrcp} that, for a given $\mmO$ and $k$,  
 \be
 \label{rqrcp_rrqr_bound}
     1  \leq \frac{\ms_{max}(\bR_{22})}{\ms_{k+1}(\bM)} \leq C_{\me}(1+\sqrt{\frac{1+\me}{1-\me}})^k \sqrt{(k+1)(n-k)}.
\ee
Stronger guarantees can be obtained by applying strong rank-revealing QR to the matrix $\bM$ after this initial step, as described in \cite{rqrcp}. While this reduces the number of permutations compared to applying strong rank-revealing QR directly to $\bM$, it still incurs additional computational and communication costs.

In this paper we study a spectrum-revealing QR factorization with column interchanges which uses as pivots the columns selected by the strong rank-revealing QR factorization of $\bM^{sk}$. We refer to this factorization as the randomized strong rank-revealing QR factorization.  
%In \cite{srrqr}, it turns out that the strong RRQR factorization can be obtained once the following important inequality holds:
%\be \rho(\bR,k) \coloneqq \sqrt{\max\limits_{1\leq i \leq k,\ 1\leq j \leq n-k}(\bR_{11}^{-1} \bR_{12})_{i,j}^2 + \mw_i^2(\bR_{11}) \mr_j^2(\bR_{22})} \leq f, \ee
%where $f>1$ is a parameter, $\mw_i(\bA)$ is the $2$-norm of the $i$th row of $\bA^{-1}$, $\mr_j(\bA)$ is the $2$-norm of $j$th column of $\bA$. Some relevant works(for example \cite{carrqr}) let an approximation of the quantity $\rho(\bR,k)$ become smaller than $f$, and obtain a factorization with weaker RRQR properties, while we consider the essential form of this quantity instead:
To demonstrate that it satisfies the strong rank-revealing property from \eqref{def_srrqr}, we use the fact, as shown in \cite{srrqr}, that this property holds if
\be
\max\limits_{1\leq i\leq k,1\leq j\leq n-k} \lab \operatorname{det}(\bar{\bR}_{11})/\operatorname{det}(\bR_{11}) \rab \leq \tilde{f},
\ee
where $\bar{\bR}_{11}$ is the leading submatrix of the $\bR$ factor obtained after interchanging the $i$-th and the $(j+k)$-th columns of $\bM$, and in our case, $\tilde{f} = \sqrt{\frac{1+\me}{1-\me}}f$. This quantity has a geometric interpretation. It is the ratio of the lengths of two vectors, which are the residuals of two least squares problems of the form $\min\limits_{\bx\in\R^{k - 1}}\|\bA\bx - \bb_1\|_2$ and $\min\limits_{\bx\in\R^{k - 1}}\|\bA\bx - \bb_2\|_2$, where $\bA \in \R^{m \times {(k-1)}}$ and $\bb_1$, $\bb_2 \in \R^m$. This observation allows us to prove the following result. Let $\bM_k,\ \bM_k^{(i,j)}\in \R^{m \times k}$ be two submatrices of $\bM$ that are identical, except for a single column. Assume that $\mmO$ is an $\varepsilon$-subspace embedding of the ranges of $\bM_k$ and $\bM_k^{(i,j)}$. Consider the $\bR$ factors obtained from their thin QR factorizations, $\bR_{11}\in \R^{k \times k}$ and $\bar{\bR}_{11}^{(i,j)}\in\R^{k \times k}$, respectively. Then the ratio of their determinants is preserved through sketching, that is
\be \label{ieq_ratio_intro}
\sqrt{\frac{1-\me}{1+\me}} \left|\frac{\operatorname{det}(\bar{\bR}_{11}^{sk\  (i,j)})}{\operatorname{det}(\bR_{11}^{sk})}\right| \leq \lab \frac{\operatorname{det}(\bar{\bR}_{11}^{(i,j)})}{\operatorname{det}(\bR_{11})} \rab \leq  \sqrt{\frac{1+\me}{1-\me}}\lab\frac{\operatorname{det}(\bar{\bR}_{11}^{sk\ (i,j)})}{\operatorname{det}(\bR_{11}^{sk})}\rab,
\ee
where ${\bR}_{11}^{sk}$ and $\bar{\bR}_{11}^{sk\ (i,j)}$ are the $\bR$ factors of the sketches $\mmO \bM_k$ and $\mmO \bM_k^{(i,j)}$. This allows us to show that, for a given $k$, the randomized strong RRQR factorization effectively reveals the spectrum of $\bM$ by satisfying the strong rank-revealing property from \eqref{def_srrqr}, that is 

\begin{comment}
\begin{equation}
\begin{split}
  1  \leq  \frac{\ms_i(\bM)}{\ms_i(\bR_{11})} \leq \sqrt{1 + \tilde{f}^2 k (n - k) } \;  &\ \text{and }\ \;1  \leq \frac{\ms_j(\bR_{22})}{\ms_{j+k}(\bM)} \leq \sqrt{1 + \tilde{f}^2 k (n - k) }, \\
  |(\bR_{11}^{-1} \bR_{12})_{i,j}| &\leq \tilde{f}\coloneqq \sqrt{\frac{1+\me}{1-\me}}f,
\end{split}
\end{equation}
\end{comment}

\begin{equation}
\label{bounds_randsrrqr}
\begin{split}
  1  \leq  \frac{\ms_i(\bM)}{\ms_i(\bR_{11})} \leq \sqrt{1 + \frac{1+\me}{1-\me} f^2 k (n - k) } \;  \ \text{and }&\ \;1  \leq \frac{\ms_j(\bR_{22})}{\ms_{j+k}(\bM)} \leq \sqrt{1 + \frac{1+\me}{1-\me} f^2 k (n - k) }, \\
  |(\bR_{11}^{-1} \bR_{12})_{i,j}| &\leq  \sqrt{\frac{1+\me}{1-\me}}f,
\end{split}
\end{equation}
for any $1\leq i \leq k$, $1\leq j \leq n-k$, where $f$ is a small constant, and assuming that $\ms_{j+k}(\bM) \neq 0$ to improve readability (see Remark \ref{remark} for the case when $\ms_{j+k}(\bM) = 0$). This property holds when the range of $\bM$ can be effectively embedded, as for tall and skinny matrices when $m \gg n$ or for low-rank matrices. It also holds for general matrices when $k \ll n$ (see Theorems \ref{th_rsrrqr1} and \ref{th_rsrrqr_square}). In the former case and when $k$ is not known, we present a randomized algorithm that computes a strong RRQR such that the maximum norm of each column of $R_{22}$ is smaller than $\frac{\tau}{\sqrt{1 - \me}}$, where $\tau$ is a given tolerance, and $\|R_{22}\|_2 \leq \sqrt{\frac{n - k}{1 - \me} } \tau$. 
%Related results are presented, for instance, in Theorem 5 of \cite{cqrrpt}, although under more restrictive assumptions and with looser bounds, particularly regarding $\bR_{11}^{-1} \bR_{12}$.

The randomized strong RRQR factorization studied in this paper can be applied to solve different problems. It can be used to select $k$ columns to compute the low-rank approximation of a matrix by using QR with column interchanges \cite{ming_rqrcp,hqrrp,rqrcp},  interpolative decompositions, CUR, or QLP factorizations (see e.g. \cite{Voronin2017_book_curid} for interpolative decompositions and CUR, see \cite{qlp} for QLP).  The $\bR_{11}$-factor of the sketch of $\bM$ computed by strong RRQR can be used as a preconditioner for solving least-squares problems as in \cite{rokhlin_fast_2008} or for orthogonalizing a set of vectors as in preconditioned CholeskyQR (see e.g. \cite{blocked_RGS,balabanov2022randomizedcholeskyqrfactorizations, garrison2024randomizedpreconditionedcholeskyqralgorithm, cqrrpt}). In general, while the randomized strong RRQR discussed in this paper is closely related to previous approaches, our main contribution consists in showing that such a factorization satisfies the strong rank-revealing properties as given in \eqref{bounds_randsrrqr}. These results can be used to improve the theoretical guarantees of other algorithms.

The paper is organized as follows. Section \ref{section_pre} reviews the deterministic strong rank-revealing QR factorization and random sketching. Section \ref{section_sketch_geo} discusses several relevant properties preserved through sketching, as angles between two vectors, angles between a vector and a subspace, and ratio of volumes of tall and skinny matrices that are identical except for one column.  Section \ref{section_rrqr} introduces randomized strong RRQR algorithms by considering two cases, when the dimension of the range of $\bM$ is much smaller than $m$, or when $k \ll m$. In the former case, we also provide two algorithms, when either a rank $k$ or a tolerance $\tau$ are given. We prove strong rank-revealing properties of the obtained factorizations. Section \ref{section_tests} provides numerical experiments showing the effectiveness of the proposed approaches.
 
\section{Preliminaries}
\label{section_pre}
In this section, we first discuss the strong RRQR factorization and its properties, and then we introduce random sketching, the $\epsilon$-embedding property, the oblivious subspace embedding property, and the usage of randomization for solving least squares problems.
\subsection{Notations}
Consider the partial QR factorization of a matrix $\bM$,
\be \label{QR}
\bM = \bQ \bR = \bQ \begin{pmatrix}
	\bR_{11}       & \bR_{12}\\
	\bO    & \bR_{22}
\end{pmatrix},
\ee
where $\bR_{11} \in \R^{k \times k}$ is an upper-triangular matrix, $\bR_{12} \in \R^{k\times (n-k)}$ and $\bR_{22} \in \R^{(m-k)\times(n-k)}$ are general matrices. We denote the partial $\bR$ factor of this factorization as $\cR_k(\bM) = \begin{pmatrix}
		\bR_{11}       & \bR_{12}\\
		\bO    & \bR_{22}
\end{pmatrix}$. For an invertible matrix $\bA \in \R^{m \times m}$ and a general matrix $\bB \in \R^{m \times n}$, we refer to the norm of the $i$th row of $\bA^{-1}$ as $\omega_i(\bA)$ and to the norm of the $j$th column of $\bB$ as $\mr_j(\bB)$. We also use the notation $\mw_*(\bA) = \bigg( \omega_1(\bA),\; \omega_2(\bA),\; \cdots \;\omega_m(\bA)\bigg)$ and $\mr_*(\bB) = \bigg( \mr_1(\bB),\; \mr_2(\bB),\; \cdots \;\mr_n(\bB) \bigg)$.  

In the context of the strong RRQR factorization we use the notation 
 \be
 \label{def_rho}
  \rho(\bR,k) :=  \max\limits_{1\leq i \leq k,\ 1\leq j \leq n-k}\sqrt{(\bR_{11}^{-1} \bR_{12})_{i,j}^2 + \mw_i^2(\bR_{11})\mr_j^2(\bR_{22})},
 \ee
 and its upper bound $\sqrt{2}\hat{\rho}(\bR,k)$, where
\[ \hat{\rho}(\bR,k) :=
 \max\bigg(\max\limits_{1\leq i \leq k,\ 1\leq j \leq n-k}|(\bR_{11}^{-1} \bR_{12})_{i,j}|,\max\limits_{1\leq i \leq k,\ 1\leq j \leq n-k}\mw_i(\bR_{11})\mr_j(\bR_{22}) \bigg).
\] 
We refer to the permutation matrix that interchanges the $i$-th and $j$-th columns of a matrix as $\mmP_{i,j}$. The factorizations described in this paper often require interchanging one of the first $k$ columns of $\bM$ with one of the last $n - k$ columns, where $\ijk$. They also require recomputing a QR factorization of the permuted matrix $\bM \mmP_{i,j+k}$. In this context,   $\bar{\bR}^{(i,j)}$ denotes the new partial $\bR$ factor and $\bar{\bR}_{11}^{(i,j)}\in \R^{k\times k}$, $\bar{\bR}_{12}^{(i,j)}\in \R^{k\times (n-k)}$, $\bar{\bR}_{22}^{(i,j)}\in \R^{(m-k) \times (n-k)}$ denote its submatrices,  $i.e\ \bar{\bR}^{(i,j)} = \cR_k(\bM \mmP_{i,j+k}) = \begin{pmatrix}
    \bar{\bR}_{11}^{(i,j)} & \bar{\bR}_{12}^{(i,j)} \\
    \bO  & \bar{\bR}_{22}^{(i,j)}
\end{pmatrix}$.
	
\subsection{Strong rank-revealing QR factorization}

One of the most used algorithms for selecting linearly independent columns from $M$ and revealing its rank is
QR with column pivoting (QRCP) \cite{Golub1965_2}.  At each iteration, it selects the column of maximum norm from the trailing matrix $\bR_{22}$ to increase the value of $\ms_{min}(\bR_{11})$ and decrease that of $\ms_{max}(\bR_{22})$.  Randomized versions of this algorithm are presented in \cite{rqrcp,hqrrp}.  They are based on a block algorithm that selects at each iteration column pivots using the QRCP factorization of $\bM^{sk} = \mmO \bM$, where $\mmO$ is a random sketching matrix.

The strong rank-revealing QR factorization \cite{srrqr} aims to approximate the leading $k$ singular values of $\bM$ by those of $\bR_{11}$ and the remaining $\min(m,n)-k$ by those of $\bR_{22}$. Consider that $k$ is given. Since 
	$$
	\lab \operatorname{det}(\bR_{11}) \rab = \prod\limits_{i=1}^k \ms_i(\bR_{11}),
	\lab \operatorname{det}(\bR_{22}) \rab = \prod\limits_{j=1}^{n-k} \ms_j(\bR_{22}),
	\lab \operatorname{det}(\bR)\rab = \lab \operatorname{det}(\bR_{11}) \operatorname{det}(\bR_{22})\rab.
	$$
the main idea of this algorithm is to maximize the absolute value of $\operatorname{det}(\bR_{11})$ by repeatedly selecting and interchanging two columns, so that the singular values of $\bR_{11}$ increase while the singular values of $\bR_{22}$ decrease.
At each step of the algorithm, the $i$-th and the $j+k$-th columns are selected and interchanged such that $|\operatorname{det}(\bar{\bR}_{11}^{(i,j)})/\operatorname{det}(\bR_{11})|>f$, where $f>1$ is a given parameter and $\operatorname{det}(\bar{\bR}_{11}^{(i,j)})$ is the leading submatrix obtained after interchanging the $i$-th and the $(j + k)$-th columns.  Thus $\operatorname{det}(\bR_{11})$ increases by at least a factor of $f$ after each column interchange. After $p$ column interchanges, the determinant increases as $f^p$, hence it can quickly get close to (or attain) a local maximal value.

The algorithm stops when there does not exist such a pair $(i,j)$. To improve the efficiency of the algorithm, it is shown in Lemma 3.1 of \cite{srrqr} that 
\begin{equation}
\lab \operatorname{det}(\bar{\bR}_{11}^{(i,j)})/\operatorname{det}(\bR_{11}) \rab = \sqrt{(\bR_{11}^{-1} \bR_{12})_{i,j}^2 + \mw_i^2(\bR_{11}) \mr_j^2(\bR_{22})}.
\end{equation}
To ensure that $\lab \operatorname{det}(\bar{\bR}_{11}^{(i,j)})/\operatorname{det}(\bR_{11}) \rab<f$, it is sufficient that both $\max\limits_{i,j} \lab (\bR_{11}^{-1} \bR_{12})_{i,j} \rab $ and  $\max\limits_{i,j}\mw_i(\bR_{11})\mr_j(\bR_{22})$ are smaller than $\frac{\sqrt{2}}{2}f$. Thus the pair $(i,j)$ can be determined by either $\argmax\limits_{i,j}\lab (\bR_{11}^{-1} \bR_{12})_{i,j} \rab $ or $\argmax\limits_{i,j}\mw_i(\bR_{11})\mr_j(\bR_{22})$.

Algorithm \ref{alg:SRRQR_efficient} presents this strong RRQR factorization, slightly modified from \cite{srrqr}.  The rank $k$ is determined by progressively increasing the size of $\bR_{11}$ while satisfying the rank-revealing properties. The formulas for updating and modifying some required quantities are provided in \cite{srrqr}. 

\begin{algorithm}[htbp]
\caption{Compute a $k$ and a strong RRQR factorization(SRRQR)}\label{alg:SRRQR_efficient}
\begin{algorithmic}[1]
  \Require $\bM$ is $m \times n$ matrix, parameter $f>1$.
  \Statex \quad \quad \textbf{Options:} $r$ is target rank to be revealed, $\tau$ is parameter such that $\max\limits_{1\leq j \leq n-k} \gamma_j(\bR_{22}) \leq \tau$. 
  \Statex \quad \quad (Either $r$ or $\tau$ should be set.)
  \Ensure $k$ is size of $\bR_{11}$, $\bQ$ is $m \times m$ orthogonal matrix, $\bR$ is $m \times n$ upper trapezoidal matrix, $\mmP$ is $n \times n$ permutation matrix such that $\bM\mmP = \bQ \bR$.
  \Function{[$k$,$\bQ, \bR, \mmP$] = SRRQR}{$\bM,f,r,\tau$}
   \State $k \coloneqq 0$, $\bR \coloneqq \bM$, $\mmP \coloneqq \bI_n$
   \State Initialize $\omega_*(\bR_{11}), \gamma_*(\bR_{22})$ and $\bR_{11}^{-1}\bR_{12}$ 
   \While{ $\max\limits_{1\leq j \leq n-k} \gamma_j(\bR_{22}) \geq \tau$ (if $\tau$ is given) or $ k \leq r$ (if $r$ is given)}
        \State $j_{max} \coloneqq \arg \max_{1\leq j \leq n-k} \gamma_j(\bR_{22}) $
        \State $k \coloneqq k + 1$
        \State Compute $\bar{\bR}^{(k,j_{max}-1)} = \begin{pmatrix}
        \bar{\bR}_{11}^{(k,j_{max}-1)}       & \bar{\bR}_{12}^{(k,j_{max}-1)}\\ 
					\bO    & \bar{\bR}_{22}^{(k,j_{max}-1)} 
				\end{pmatrix} \coloneqq \cR_k(\bR\mmP_{k,k+j_{max-1}})$ and $\mmP \coloneqq \mmP \mmP_{k,k+j_{max-1}}$
        \State Update $\omega_*(\bR_{11}), \gamma_*(\bR_{22})$ and $\bR_{11}^{-1}\bR_{12}$
        \While{$\hat{\rho}(\bR,k) > f/\sqrt{2}$}
			\State Find $i$ and $j$ such that $|(\bR_{11}^{-1}\bR_{12})_{i,j}|>f$ or $\omega_i(\bR_{11})\gamma_j(\bR_{22}) > f$
            \State Compute $\bar{\bR}^{(i,j)} = \begin{pmatrix}
					\bar{\bR}_{11}^{(i,j)}       & \bar{\bR}_{12}^{(i,j)}\\ 
					\bO    & \bar{\bR}_{22}^{(i,j)} 
				\end{pmatrix} \coloneqq \cR_k(\bR \mmP_{i,j+k})$ and $\mmP \coloneqq \mmP\  \mmP_{i,j+k}$
            \State Modify $\omega_*(\bR_{11}), \gamma_*(\bR_{22})$ and $\bR_{11}^{-1}\bR_{12}$ 
			\EndWhile
   \EndWhile
   \State \Return $k,\bQ, \bR, \mmP$
  \EndFunction
		\end{algorithmic}
	\end{algorithm}
 \noindent
    
It is shown in \cite{srrqr} that the factorization obtained from Algorithm \ref{alg:SRRQR_efficient} is indeed a strong RRQR factorization. We introduce this result, which is used later in section \ref{section_rrqr} to show that the rank-revealing property is preserved in the randomized version of this factorization.
	
\begin{theorem} [Existence of a strong RRQR factorization 
\cite{srrqr}] \label{th_srrqr1}
		
Assume that we have a partial QR factorization $\bM \mmP = \bQ \bR$ and $\bR = \begin{pmatrix}
			\bR_{11}       & \bR_{12}\\
			\bO    & \bR_{22}
		\end{pmatrix}$. If $\rho(\bR,k) \leq f$, then 
		\be 
  1  \leq  \frac{\ms_i(\bM)}{\ms_i(\bR_{11})}\; , \; \frac{\ms_j(\bR_{22})}{\ms_{j+k}(\bM)} \leq \sqrt{1 + f^2 k (n-k)},\;
 |(\bR_{11}^{-1} \bR_{12})_{i,j}| \leq f,
		\ee
		$\forall 1\leq i \leq k,\ 1\leq j\leq n-k$.
	\end{theorem}

\begin{remark}
\label{remark}
    Theorem \ref{th_srrqr1} gives bounds on the ratios $\frac{\ms_i(\bM)}{\ms_i(\bR_{11})}$ and $\frac{\ms_j(\bR_{22})}{\ms_{j+k}(\bM)}$ for better readability.  For this, we assume that $\ms_{j+k}(\bM)\neq 0$, which implies that $\operatorname{rank}(\bM) \geq k$ and $\operatorname{det}(\bR_{11})\neq 0$. Thus, the denominator of $\frac{\ms_i(\bM)}{\ms_i(\bR_{11})}$ is also nonzero. If $\ms_{j_0 + k}(\bM) = 0$ for some $j_0$, then we have $\ms_{j_0}(\bR_{22}) = 0$. This remark applies to Theorem \ref{th_rsrrqr1} and \ref{th_rsrrqr_square} as well. 
\end{remark}
	
With our notation, $\rho(\bR,k) =  \max\limits_{1\leq i \leq k,\ 1\leq j \leq n-k}\sqrt{(\bR_{11}^{-1} \bR_{12})_{i,j}^2 + \mw_i^2(\bR_{11})\mr_j^2(\bR_{22})} $, where for each pair of columns $i,j$, the relation $\sqrt{(\bR_{11}^{-1} \bR_{12})_{i,j}^2 + \mw_i^2(\bR_{11})\mr_j^2(\bR_{22})}= |\operatorname{det}(\bar{\bR}_{11}^{(i,j)})/\operatorname{det}(\bR_{11})|$ is satisfied.  Thus, $\rho(\bR,k)$ represents the potential increase that can be achieved in $|\operatorname{det}(\bR_{11})|$. If $\rho(\bR,k) = 1$, then no column interchange can increase $|\operatorname{det}(\bR_{11})|$. When $1 \leq \rho(\bR,k) \leq f$ with $f$ close to $1$, $|\operatorname{det}(\bR_{11})|$ is nearly maximal.
 
\subsection{Random sketching}
\label{subsec:randomsketch}

We start by defining the $\me$-embedding property of a given subspace $\cV$.
\begin{definition}[$\me$-embedding]
\label{def_embedding}
Given a subspace $\cV \subset \R^m$, a sketching matrix $\mmO \in \R^{d\times m}$ is called an $\me$-embedding of $\cV$ if 
\be \label{sk_def1}
		|\langle \mmO \bx,\mmO \by\rangle - \langle \bx,\by \rangle| \leq \me \|\bx\|_2 \|\by \|_2,\ \forall \bx,\by \in \cV.
		\ee
\end{definition}
We note that \eqref{sk_def1} also implies the following inequality:
	\be \label{sk_def2}
	| \|\mmO \bx\|_2^2 - \|\bx\|_2^2 | \leq \me \|\bx\|_2^2.
	\ee

Since the subspace to be embedded is often not known in advance, the mapping $\mmO$ is constructed using probabilistic techniques to ensure, with high probability, that it satisfies the $\me$-embedding property for any fixed subspace $\cV \subset \R^m$. Such mappings $\mmO$ are known as oblivious subspace embeddings (OSEs).

 \begin{definition}
\label{def_oblivious_embeddings}
A random sketching matrix $\mmO \in \R^{d\times m}$ is an oblivious subspace embedding with parameters $(\me,\md,n)$ if it satisfies the $\me$-embedding property with probability at least $1-\md$ for any fixed $n$-dimensional subspace $\cV \subset \R^m$, that is,
\be
\mathbb{P}(\{\forall \bx,\by \in \cV, |\langle \mmO \bx,\mmO \by\rangle - \langle \bx,\by \rangle| \leq \me \|\bx\|_2 \|\by \|_2\}) \geq 1 - \md.
\ee
\end{definition}
We refer to an oblivious subspace embedding with parameters $(\me,\md,n)$ as $(\me,\md,n)$ OSE.  Different choices exist for the random sketching matrix $\mmO$. A very common choice consists in taking $\mmO$ to be a matrix whose elements are i.i.d standard Gaussian variables rescaled by $1/d$. In this case, $\mmO$ is an $(\me,\md,n)$ OSE if the sketch dimension is $d = O(\me^{-2}(n-log(\md))$ \cite{skt_david}.  The application of $\mmO$ to a dense matrix $\bM$ is easy to parallelize and can rely on dense optimized kernels. However its computational cost is relatively large, $2dmn$ flops. This cost can be decreased by using structured random transforms as the subsampled randomized Hadamard transform (SRHT), defined as
\be
\mmO = \sqrt{\frac{m}{d}} \bP \bH \bD.
\ee
Here $\bD$ is an $m \times m$ diagonal matrix whose diagonal entries are 
independent Rademacher random variables, with values $\pm 1$ taken with  equal probability. $\bH$ is an $m \times m$ Hadamard matrix and	$\bP$ is a $d \times m$ uniform row-sampling (with replacement) matrix. It costs $O(mnlog(m))$ flops to compute $\mmO \bM$. Given $\me$ and $\md$, the required sampling size is $d = \mathcal{O}(\me^{-2}(n+log(m-\md))log(n-\md))$ as shown in \cite{skt_david}.

The following theorem relates the singular values of a matrix $\bM$ to the singular values of the matrix $\bM^{sk} = \mmO \bM$ formed by the sketch of the columns of $\bM$, see e.g. \cite{rokhlin_fast_2008,skt_david,rgs, cqrrpt}. We include a proof for completeness.
\begin{theorem}
\label{th_sk_svd}
        Assume that $\mmO \in \R^{d \times m}$ is an $\me$-embedding of a subspace $\cV \ziji \R^{m}$ and let $\bM \in \R^{m \times n}$ be a matrix with $range(\bM)\ziji \cV$, $\bM^{sk} = \mmO \bM$. Then
        \be \label{ieq_prems}
        \sqrt{1 - \me}\  \ms_i(\bM) \leq \ms_i(\bM^{sk}) \leq \sqrt{1 + \me}\  \ms_i(\bM), \; 1 \leq i \leq d.
        \ee
    \end{theorem}
    \begin{proof}
        Using the min-max theorem, the singular values of $\bM$ are expressed as follows:
        \be
        \ms_i(\bM) = \min\limits_{\cU \text{ subspace of } \R^n \atop dim(\cU) = n+1-i} \max\limits_{\bx \in \cU \atop \|\bx\|_2 = 1} \|\bM \bx\|_2.
        \ee
        Therefore, we have
        \be
        \ms_i(\bM^{sk}) = \min\limits_{\cU \text{ subspace of } \R^n \atop dim(\cU) = n+1-i} \max\limits_{\bx \in \cU \atop \|\bx\|_2 = 1} \|\bM^{sk} \bx\|_2 \leq \min\limits_{\cU \text{ subspace of } \R^n \atop dim(\cU) = n+1-i} \max\limits_{\bx \in \cU \atop \|\bx\|_2 = 1} \sqrt{1 + \me}\  \|\bM \bx\|_2 = \sqrt{1 + \me}\  \ms_i(\bM).
        \ee
        A similar reasoning can be used to establish the lower bound.
    \end{proof}
    
Random sketching is often used to approximate the solution of a least squares problem $\min\limits_{\bx\in \R^n}\|\bM\bx - \bb\|_2$, where $\bM \in \R^{m\times n},\; \bb\in \R^m$. The following lemma (see Theorem 12 of \cite{sketch_lsq} or Section 10.2 of \cite{Martinsson_Tropp_2020}) shows that the least squares residual can be nearly preserved by random sketching. For completeness, we include its proof in this paper. 

\begin{lemma}[Sketched least squares problem] \label{lma1}
Given a matrix $\bM\in \Rmn$ and a vector $\bb\in\R^n$, assume that $\mmO\in \Rdm$ is an $\me$-embedding of $range([\bM, \bb])$, $i.e.\ \mmO$ satisfies $\eqref{sk_def1}$ for all vectors from $range([\bM, \bb])$. Let $x^{*} = \argmin\limits_{x\in \R^n} \| \bM\bx-\bb\|_2$, $\hat{\bx} = \argmin\limits_{x\in \R^n}\| \mmO ( \bM\bx-\bb )\|_2$. We have: 
\be \label{sk_lsp}
\frac{1}{\sqrt{1+\me}}\|\mmO ( \bM\hat{\bx}-\bb )\|_2 \leq \| \bM\bx^*-\bb\|_2 \leq \frac{1}{\sqrt{1-\me}}\|\mmO ( \bM\hat{\bx}-\bb )\|_2.
\ee
\end{lemma}

% \begin{lemma}[sketched least square problem] \label{lma1}
% Suppose that $\bA \in \R^{m\times k}$, $\bb \in \R^m$, $ range(\bA) \subset \cV$, $\bb \in \cV$, $\cV$ is a $n-$dimensional subspace of $\R^m$, $x^{*} = \argmin\limits_{x\in \R^n} \| \bA\bx-\bb\|_2$, $\mmO \in \R^{d\times m}$ is a $(\me,\md,n)$-OSE, $\hat{\bx} = \argmin\limits_{x\in \R^n}\| \mmO ( \bA\bx-\bb )\|_2$, then with probability at least $1-\md$ we have 
% 		\be \label{sk_lsp}
% 		\frac{1}{\sqrt{1+\me}}\|\mmO ( \bA\hat{\bx}-\bb )\|_2 \leq \| \bA\bx^*-\bb\|_2 \leq \frac{1}{\sqrt{1-\me}}\|\mmO ( \bA\hat{\bx}-\bb )\|_2
% 		\ee
% 	\end{lemma}
\begin{proof}
Since $\bx^*$ is the minimizer of the original least squares problem, by applying \eqref{sk_def2} we obtain:
\be
\| \bM\bx^*-\bb\|_2 = \min\limits_{x\in \R^n} \| \bM\bx - \bb\|_2 \leq \|\bM\hat{\bx}-\bb \|_2 \leq \frac{1}{\sqrt{1-\me}}\|\mmO ( \bM\hat{\bx}-\bb )\|_2.
\ee
Similarly, since $\hat{\bx}$ is the minimizer of the sketched least squares problem, we have:
\be
\|\mmO ( \bM\hat{\bx}-\bb )\|_2 = \min\limits_{x\in \R^n} \|\mmO ( \bM\bx-\bb )\|_2 \leq 
\|\mmO ( \bM\bx^*-\bb )\|_2 \leq \sqrt{1+\me}\| \bM\bx^*-\bb\|_2.
\ee
\end{proof}
 
\section{Random sketching, angles, volumes, and determinants}
\label{section_sketch_geo}

In this section, we study geometric properties that can be nearly preserved by an $\varepsilon$-embedding of a subspace $\cV \subset \R^m$. We first consider the angle between two vectors, then we extend the result to the angle between a vector and a subspace of $\cV$. We finally generalize the result to the volume of a parallelotope, corresponding to the area of a parallelogram in the two-dimensional (2D) case and the volume of a parallelepiped in the three-dimensional (3D) case.  

%The $\epsilon$-embedding property allows to preserve the norm of a vector with distortion $\epsilon$, as described in \eqref{sk_def2}. We aim to see how other geometric quantities can be preserved by a sketching matrix. We firstly consider the angle \cite{strang09} between two vectors $\mma$ and $\mmb$ from $\cV$. 
	
\medskip
\begin{definition}[see, e.g. Page 154 of \cite{strang2006linear}]
The angle between two nonzero vectors $\bv_1$ and $\bv_2$ is defined as:
\be \label{ang_vv}
\ang \langle \bv_1,\bv_2 \rangle = arccos(\frac{\langle \bv_1,\bv_2 \rangle}{\|\bv_1\|_2 \|\bv_2\|_2}).
\ee
\end{definition} 

\begin{lemma}
Assume that the sketching matrix $\mmO \in \bR^{d \times m}$ is an $\me$-embedding of the subspace $\cV \subset \bR^m$, $i.e.\ $the inequality \eqref{sk_def1} holds. Consider any two nonzero vectors $\bv_1,\ \bv_2\in \cV$, and let $\bv_1^{sk} = \mmO \bv_1,\ \bv_2^{sk} = \mmO \bv_2$. The following holds:
\be \label{ieq_ang_vv}
\frac{cos \ang \langle \bv_1,\bv_2 \rangle - \me}
{1+\me}\leq cos \ang \langle \bv_1^{sk}, \bv_2^{sk} \rangle \leq \frac{cos \ang \langle \bv_1,\bv_2 \rangle + \me}{1-\me}.
\ee
\end{lemma}
	
\begin{proof}
We have that
\[
cos \ang \langle \bv_1^{sk}, \bv_2^{sk} \rangle = \frac{\langle \mmO \bv_1, \mmO \bv_2 \rangle}{\| \mmO \bv_1 \|_2 \| \mmO \bv_2 \|_2}.
\] 
From the $\varepsilon$-embedding property \eqref{sk_def1} and \eqref{sk_def2}, it can be seen that the following inequalities hold:
		\[
		\langle  \bv_1,  \bv_2 \rangle -\me\|\bv_1\|_2 \|\bv_2\|_2 \leq \langle \mmO \bv_1, \mmO \bv_2 \rangle \leq \langle  \bv_1,  \bv_2 \rangle + \me \|\bv_1 \|_2 \|\bv_2 \|_2,
		\]
		\[
		(1-\me)^{\frac{1}{2}}\| \bv_1 \|_2 \leq
		\| \mmO \bv_1 \|_2 \leq
		(1+\me)^{\frac{1}{2}} \| \bv_1\|_2, 
		\]
		\[
		(1-\me)^{\frac{1}{2}}\| \bv_2 \|_2 \leq
		\| \mmO \bv_2 \|_2 \leq
		(1+\me)^{\frac{1}{2}} \| \bv_2\|_2 .
		\]
Thus we obtain the following upper bound:
\[cos \ang \langle \bv_1^{sk}, \bv_2^{sk} \rangle \leq  \frac{\langle  \bv_1,  \bv_2 \rangle + \me \|\bv_1 \|_2 \|\bv_2 \|_2}{(1-\me)\|\bv_1\|_2 \|\bv_2\|_2} = \frac{cos \ang \langle \bv_1,\bv_2 \rangle + \me}{1-\me}. \] 
%Similarly, we can get $cos \ang \langle \mma^{sk}, \mmb^{sk} \rangle \geq \frac{cos \ang \langle \pmb{\alpha},\pmb{\beta} \rangle - \me}{1+\me}$.
The lower bound is obtained by a similar reasoning.  
\end{proof}
We now consider the angle between a vector $\bx$ and a subspace $\cK$, we first recall its definition from \cite{golub_angle_space}.
\begin{definition}
\label{angle_vecspace}
The angle between a nonzero vector $\bv$ and a subspace $\cK$ is defined as 
\be \label{ang_vsp}
		\ang \langle \bv, \cK \rangle = arccos( \max\limits_{\bu\in \cK, \|\bu\|_2 \neq \bO}\frac{\langle \bv,\bu \rangle}{\|\bv\|_2 \|\bu\|_2}).
		\ee 
\end{definition}
When $\bv \notin \cK^{\perp}$, we have the following proposition.
\begin{prop}
\label{pp31}
    The angle between a nonzero vector $\bv$ and a subspace $\cK$ is equal to the angle between $\bv$ and the orthogonal projection of $\bv$ onto $\cK$, provided that $\bv \notin \cK^{\perp}$, that is
    \be
    \label{eq_angle_lsq}
    \ang \langle \bv, \cK \rangle = \ang \langle \bv, \cP(\bv) \rangle,
    \ee
    where $\cP$ is the orthogonal projector onto $\cK$. Furthermore,
    \be
    \label{eq_triangle_lsq}
    \|\cP(\bv)\|_2 = \|\bv\|_2 cos(\ang\langle \bv, \cK \rangle ),\ \|\bv - \cP(\bv)\|_2 = \|\bv\|_2 sin(\ang\langle \bv, \cK \rangle ).
    \ee
\end{prop}
\begin{proof}
    We have that $\cP(\bv) \neq \bO$ if $\bv \notin \cK^{\perp}$. Note that \eqref{eq_triangle_lsq} can be derived from the definition of trigonometric functions if \eqref{eq_angle_lsq} holds, it is sufficient to prove that
    \be
    \label{ieq_pp31}
    \langle \bv , \frac{\cP(\bv)}{\|\cP(\bv) \|_2} \rangle \geq \langle \bv , \frac{\bu}{\|\bu \|_2} \rangle,\ \forall \bu \in \cK \ \text{and}\ \bu \neq \bO.
    \ee
    Note that $\bv = \cP(\bv) + (\bI - \cP)(\bv)$ and $\cP$ is the orthogonal projector, the left of \eqref{ieq_pp31} is equal to $\|\cP(\bv)\|_2$. Hence it is sufficient to prove that 
    \be
    \langle \bu , \bv \rangle \leq  \|\bu \|_2 \| \cP(\bv) \|_2,\ \forall \bu \in \cK \ \text{and}\ \bu \neq \bO.
    \ee
    Again using the orthogonality of $\cP$, we have
    \be
    \langle \bu,\bv \rangle = \langle \bu, \cP(\bv) \rangle  + \langle \bu , (\bI - \cP)(\bv) \rangle =  \langle \bu, \cP(\bv) \rangle.
    \ee
    The proof is finished by using the Cauchy inequality.
\end{proof}
Using the definitions \ref{def_embedding} and \ref{angle_vecspace}, we obtain the following lemma.
\begin{lemma}
Assume that the sketching matrix $\mmO \in \bR^{d \times m}$ is an $\me$-embedding of a subspace $\cV \subset \bR^m$. For any vector $\bv\in\cV$ and any subspace $\cK \subset \cV$, let $\bv^{sk} = \mmO \bv,\  \cK^{sk} = \{\bu^{sk} | \bu^{sk} = \mmO \bu, \bu\in \cK \}$. We have that
		\be
		\frac{cos \ang \langle \bv,\cK \rangle - \me}{1+\me}\leq
		cos \ang \langle \bv^{sk}, \cK^{sk} \rangle \leq
		\frac{cos \ang \langle \bv,\cK \rangle + \me}{1-\me}.
		\ee
	\end{lemma}	
\begin{proof}
Using Definition \ref{angle_vecspace} and the $\me$-embedding property from Definition \ref{def_embedding}, we obtain: 
\[
    cos \ang \langle \bv^{sk}, \cK^{sk} \rangle = \max\limits_{\bu^{sk}\in \cK,\|\bu^{sk}\|_2 \neq \bO}\frac{\langle \bv^{sk},\bu^{sk} \rangle}{\|\bv^{sk}\|_2\|\bu^{sk}\|_2} \leq \max\limits_{\bu\in \cK,\|\bu\|_2 \neq \bO} \frac{\langle \bv, \bu \rangle+\me \|\bv\|_2 \|\bu\|_2}{(1-\me)\|\bv\|_2 \|\bu\|_2} = \frac{cos\ang \langle \bv, \cK \rangle + \me }{1 - \me}. \] 
    The lower bound follows from a similar reasoning.
    \end{proof}

\begin{figure}[htbp]
\centering
    \begin{minipage}[c]{0.4\textwidth}
    \centering
    \centerline{\includegraphics[width=1.0\textwidth]{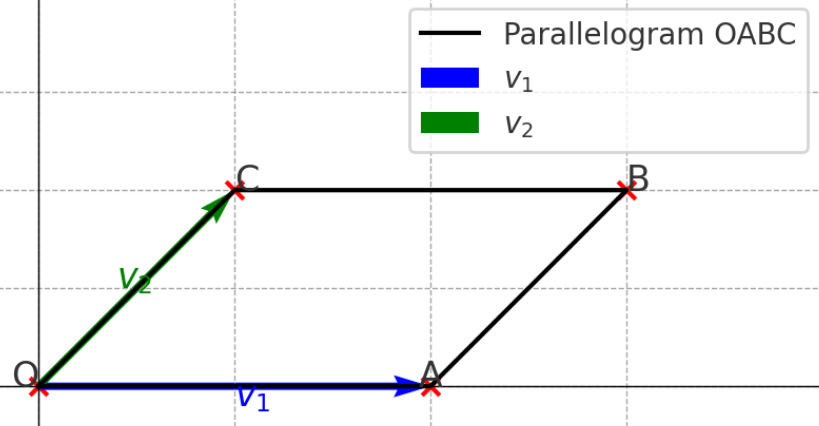}}
    \caption{parallelogram OABC}
    \label{fig:area}
    \end{minipage}
    \hfill
    \begin{minipage}[c]{0.48\linewidth}
    \centering
    \centerline{\includegraphics[width=1.0\textwidth]{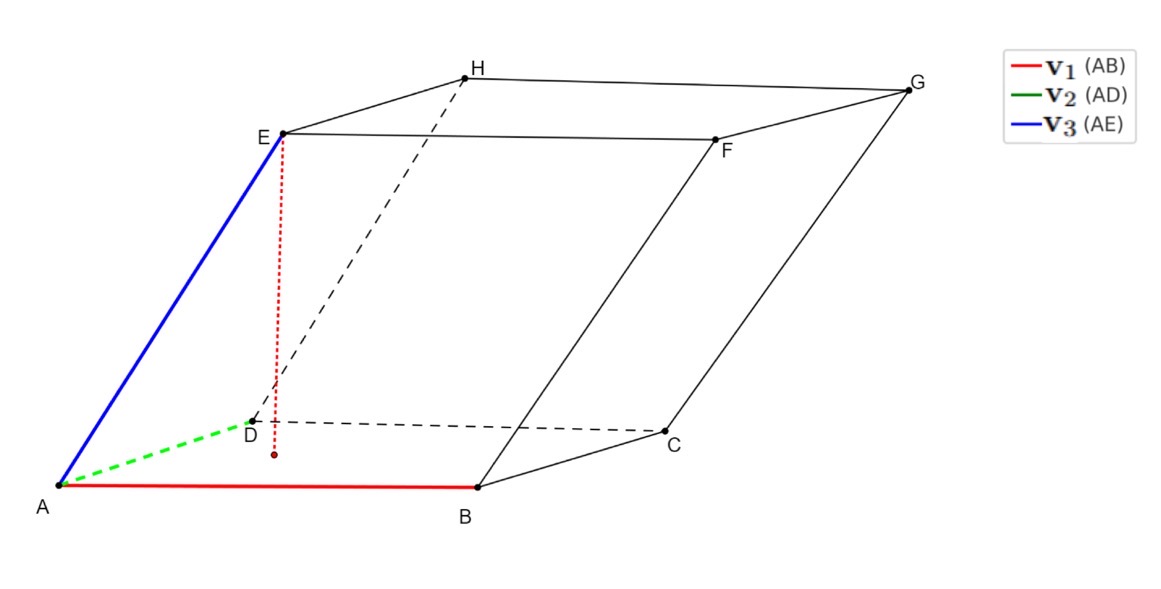}}
    \caption{volume of a $3\times3$ matrix $\bigg( \bv_1 \ ,\ \bv_2 \ ,\ \bv_3\bigg)$}
    \label{fig:volume}
    \end{minipage}
    
\end{figure}

We now discuss the geometric interpretation of the determinant of a matrix (for more details see e.g. \cite{MAALA} or \cite{strang2006linear}). Consider first a $2\times2$ matrix $\bM = \bigg( \bv_1 \ ,\ \bv_2\bigg) = \begin{pmatrix}
			m_{11}    & m_{12}\\
			m_{21}    & m_{22}
\end{pmatrix}$. The absolute value of its determinant, $|\operatorname{det}(\textbf{\bM})|$, is equal to the area of the parallelogram defined by $\bv_1$ and $\bv_2$. As illustrated in Figure \ref{fig:area}, let $OA$ denote $\bv_1$, $OC$ denote $\bv_2$, and let $\parallelogram OABC$ represent the parallelogram they define. Indeed, its area is:
\[
\begin{aligned}
    S_{\parallelogram OABC} &= \|OA\|_2 \|OC\|_2 sin\ang \langle OA, OC \rangle \\
             &= \|OA\|_2 \|OC\|_2 \sqrt{1 - cos^2\ang \langle OA, OC \rangle} \\
             &= \sqrt{\|OA\|_2^2 \|OC\|_2^2 - \langle OA,OC \rangle^2} \\
             &= \sqrt{(m_{11}^2 + m_{21}^2)(m_{12}^2 + m_{22}^2) - (m_{11}m_{12}+m_{21}m_{22})^2} \\
             &= |m_{11} m_{22} - m_{21} m_{12}| = |\operatorname{det}(\bM)|.
\end{aligned}
\]
In the case of a $3 \times 3$ matrix $\bM = \bigg( \bv_1 \ ,\ \bv_2 \ ,\ \bv_3\bigg)$, $|\operatorname{det}(\bM)|$ is equal to the volume of the parallelepiped formed by $\bv_1, \bv_2, \bv_3$. As illustrated in Figure \ref{fig:volume}, the volume is $V = S_{\parallelogram ABCD}\ h$, where $h$ is the height of the parallelepiped.  This formula can be generalized to larger dimensions as in the following lemma, which will be used later to relate the determinant of a matrix to the determinant of the matrix formed by the sketch of its columns. For a general matrix $\bM \in \R^{m\times n}$, \bM = $\bigg( \bv_1 \ , \ \bv_2\ ,\  \cdots \ , \ \bv_n\bigg)$, we refer to the volume of the $n-$dimensional box formed by its columns as $V(\bM)$, and we have: 
\be
V(\bM)= \sqrt{\operatorname{det}(\bM^T\bM)} = |\operatorname{det}(\textbf{R})| = \prod_{i=1}^{\min\{m,n\}} \ms_i(\bM),
\ee
where \textbf{R} is the square R factor obtained from the thin QR factorization of \bM, $i.e.\ \bM = \textbf{QR}, \textbf{Q}\in \R^{m\times n},\ \textbf{R}\in \R^{n \times n} $.

\begin{lemma}
\label{lma_volume}
Given a matrix $\bM \in \R^{m\times n}$, \bM = $\bigg( \bv_1 \ , \ \bv_2\ ,\  \cdots \ \bv_n\bigg)$, $\forall 1\leq i\leq n$, let $\bM_{n-1}\in\R^{m\times (n-1)}$ denote the matrix formed by removing the $i$-th column of $\bM$, that is $\bM_{n-1} = \bigg( \bv_1 \ , \cdots , \ \bv_{i-1}\ , \bv_{i+1} \  \cdots \ \bv_n\bigg)$. We have
\begin{eqnarray}
\label{eq_volume}
\nonumber
V(\bM) &=& V(\bM_{n-1}) \|\bv_i - \mathcal{P}(\bv_i)\|_2 = V(\bM_{n-1})\min\limits_{\bx \in \R^{n-1}}\| \bv_i - \bM_{n-1}\bx \|_2\\
&=& V(\bM_{n-1}) \|\bv_i\|_2 sin \ang \langle \bv_i , range({\bM}_{n-1}) \rangle,
\end{eqnarray} 
where $\mathcal{P}$ is the orthogonal projector onto the subspace $range(\bM_{n-1})$.
\end{lemma}
\begin{proof}
Since $V(\bM) = V(\bM \mmP)$ for any permutation matrix $\mmP$, it is sufficient to consider the case $i = n$. The QR factorization of $\bM$ can be written as 
		\be
		\bM = \bQ \bR = \bQ \begin{pmatrix}
			\bR_{n-1} & \mmr \\
			\bO & r_{n,n}
		\end{pmatrix},
		\ee
		where $r_{n,n}$ is the last diagonal element of $\bR$. We have that $V(\bM) = |\operatorname{det}(\bR)| = |\operatorname{det}(\bR_{n-1})r_{n,n}| =  V(\bM_{n-1}) |r_{n,n}|.$ Let $\bQ = \bigg( \bq_1 \ , \ \bq_2\ ,\  \cdots \ \bq_n\bigg)$ and $\bQ_{n-1}\in\R^{m\times(n-1)}$ denote the matrix composed of the first $k-1$ columns of $\bQ$, the orthogonal projector $\cP$ is given by $\bQ_{n-1}\bQ_{n-1}^T$. Since $v_n = r_{1,n}\bq_1 + r_{2,n}\bq_2 + \cdots + r_{n,n}\bq_n$ and $\bQ_{n-1}\bQ_{n-1}^Tv_n = r_{1,n}\bq_1 + r_{2,n}\bq_2 + \cdots + r_{n-1,n-1}\bq_{n-1}$, we obtain
\begin{eqnarray}
\label{eq_volume1}
|r_{n,n}| = \|r_{n,n}\bq_n\|_2 = \|\bv_n - \bQ_{n-1}\bQ_{n-1}^T\bv_n \|_2 = \min\limits_{\bx \in \bR^{n-1}} \|\bv_n - \bM_{n-1}\bx \|_2.
% &=&  \| \bv_n \|_2 sin\ang \langle \bv_n, range({\bM}_{n-1}) \rangle.
\end{eqnarray}
In addition, if $\bv_n \in range({\bM}_{n-1})^{\perp}$,  then $\ang \langle \bv_n , range({\bM}_{n-1}) \rangle = \frac{\pi}{2}$ we have 
\be
\|\bv_n - \bQ_{n-1}\bQ_{n-1}^T\bv_n \|_2 = \|\bv_n\|_2 sin \ang \langle \bv_n , range({\bM}_{n-1}) \rangle.
\ee
If $\bv_n \notin range({\bM}_{n-1})^{\perp}$ , using Proposition \ref{pp31}, we can also get
\be
\|\bv_n - \bQ_{n-1}\bQ_{n-1}^T\bv_n \|_2 = \| \bv_n\|_2  sin \ang \langle \bv_n , range({\bM}_{n-1}) \rangle.
\ee
\end{proof}

\begin{figure}[htbp]
    \centering
    \includegraphics[width=0.5\textwidth]{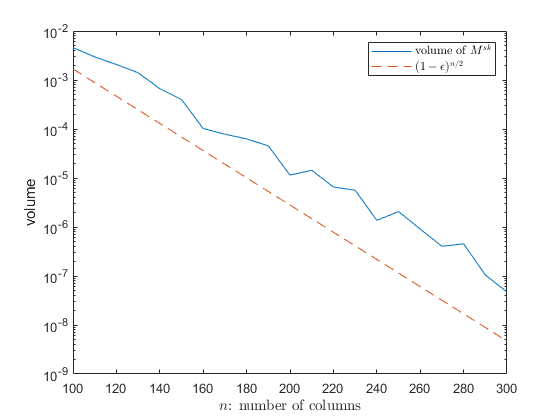}
    \caption{Volume of the sketch of an orthogonal matrix matrix obtained from sampling columns of the identity matrix.}
    \label{fig:volume_decay}
\end{figure}

We discuss now the relation between the volume of a matrix and the volume of its sketch. Let $\bM \in \bR^{m \times n}$ be a matrix formed as $\bM = (\bv_1, \ \bv_2,\  \cdots ,\ \bv_n)$, with $range(\bM) \subset \cV$, and let $\bM^{sk} \in \bM^{d \times n}$,  $\bM^{sk} = \mmO \bM = (\bv_1^{sk}, \ \bv_2^{sk},\  \cdots ,\ \bv_n^{sk})$, where $\Omega \in \bR^{d \times m}$ is an $\varepsilon$-embedding of the subspace $\cV \subset \bR^m$. Let their thin QR factorizations be $\bM = \textbf{QR}$, where $\bQ\in\Rmn,\ \bR \in \Rnn,$ and $\bM^{sk} = \textbf{Q}^{sk}\textbf{R}^{sk}$, where $\bQ^{sk}\in\Rdn,\  \bR^{sk}\in\Rnn$ respectively.  We first show that the volumes might not be preserved by sketching. In other words, the values of  $|\operatorname{det}(\textbf{R})|$ and $|\operatorname{det}(\textbf{R}^{sk})|$ can differ significantly. Indeed, using \eqref{eq_volume1}, we obtain by induction that
\begin{eqnarray}
V(\bM) &=& \prod_{i=1}^{n}\|\bv_i\|_2 \prod_{i=1}^{n-1} sin\ang \langle \bv_{i+1}, span\{\bv_1,\bv_2,\cdots,\bv_i\} \rangle, \\
V(\bM^{sk}) &=& \prod_{i=1}^{n}\|\bv_i^{sk}\|_2 \prod_{i=1}^{n-1} sin\ang \langle \bv_{i+1}^{sk}, span\{\bv_1^{sk},\ \bv_2^{sk},\cdots,\bv_i^{sk}\} \rangle.
\end{eqnarray}
In addition, we have that 
\begin{equation}
(1-\me)^{\frac{n}{2}}\prod_{i=1}^{n}\|\bv_i\|_2 \leq \prod_{i=1}^{n}\|\bv_i^{sk}\|_2 \leq (1+\me)^{\frac{n}{2}}\prod_{i=1}^{n}\|\bv_i\|_2.
\end{equation}
For sufficiently large $n$, the term $(1\pm \me)^{\frac{n}{2}}$ can have a large absolute value, and may not be sufficiently balanced by the second term $\prod_{i=1}^{n-1} sin\ang \langle \bv_{i+1}^{sk}, span\{\bv_1^{sk},\ \bv_2^{sk},\cdots,\bv_i^{sk}\}\rangle$.  An example is illustrated in Figure \ref{fig:volume_decay}. We consider an orthogonal matrix $\bW\in \R^{8192\times 500}$ obtained by sampling columns from the identity matrix of dimensions ${8192} \times 8192$, and an SRHT matrix $\mmO\in\R^{1500\times 500}$.
%, which is $\me$-embedding of $range(\bW)$ with high probability. 
We further sample uniformly at random columns from $\bW$ to form a matrix $\bM$, whose volume $V(\bM)$ is $1$, and let the number of columns of $\bM$ increase from $100$ to $300$. As it can be seen in the figure, the volume of $\mmO \bM$ is much smaller than $1$ and as $n$ increases, the volume is observed to decrease as $(1 - \me)^{n/2}$, where $\me$ is $1/4$. 
 
However, when computing a strong rank-revealing factorization, we are interested in the ratio of the volumes of two matrices that differ by only a single column, as it can be seen from Algorithm \ref{alg:SRRQR_efficient}. The following lemma is the key result of this section, showing that in this particular case, the ratio of the volumes of these two matrices is preserved through sketching. It will be used in the following section to derive a randomized version of the strong rank-revealing factorization.

\begin{lemma} \label{lma2}
Consider a matrix $\bM \in \R^{m\times n}$, and let the sketching matrix ${\mmO} \in \R^{d\times m}$ be an $\me$-embedding of $range(\bM)$. Let $\bM^{sk} = \mmO \bM$, for any integer k, $1 \leq k \leq n$, let $k$-partial QR factorization of $\bM^{sk}\mmP$ and $\bM \mmP$ be $\bM^{sk} \mmP = \bQ^{sk} \bR^{sk} = \bQ^{sk} \begin{pmatrix}
			\bR_{11}^{sk}     & \bR_{12}^{sk}\\
			\bO    & \bR_{22}^{sk}
		\end{pmatrix}$, $\bM \mmP = \bQ \bR = \bQ \begin{pmatrix}
			\bR_{11}     & \bR_{12}\\
			\bO    & \bR_{22}
		\end{pmatrix}$. For all $1\leq i \leq k,\ 1\leq j\leq n-k$, let $\bar{\bR}^{sk\ (i,j)} = \cR_k(\bR^{sk} \mmP_{i,j+k}) = \begin{pmatrix}
			\bar{\bR}_{11}^{sk\ (i,j)}      & \bar{\bR}_{12}^{sk\ (i,j)} \\
			\bO    & \bar{\bR}_{22}^{sk\ (i,j)} 
		\end{pmatrix}$ and $\bar{\bR} = \cR_k(\bR \mmP_{i,j+k}) = \begin{pmatrix}
			\bar{\bR}_{11}^{(i,j)}     & \bar{\bR}_{12}^{(i,j)}\\
			\bO    & \bar{\bR}_{22}^{(i,j)}
		\end{pmatrix}$. The following holds:
		\be \label{ieq_ratio}
		\sqrt{\frac{1-\me}{1+\me}} \left|\frac{\operatorname{det}(\bar{\bR}_{11}^{sk\  (i,j)})}{\operatorname{det}(\bR_{11}^{sk})}\right| \leq \lab \frac{\operatorname{det}(\bar{\bR}_{11}^{i,j})}{\operatorname{det}(\bR_{11})} \rab \leq  \sqrt{\frac{1+\me}{1-\me}}\lab\frac{\operatorname{det}(\bar{\bR}_{11}^{sk\ (i,j)})}{\operatorname{det}(\bR_{11}^{sk})}\rab,\ \forall \ijk.
		\ee
	\end{lemma}
\begin{proof}
Let $\bM\mmP = \bigg(\bv_1, \ \bv_2,\ \cdots, \ \bv_n\ \bigg)$ and $\bM^{sk}\mmP = \bigg(\bv_1^{sk}, \ \bv_2^{sk},\  \cdots,\ \bv_n^{sk}\bigg)$, where $\bv_i^{sk} = \mmO \bv_i, \ \forall 1\leq i \leq n$. Let $\bM_k = \bigg(\bv_1, \ \bv_2,\  \cdots, \ \bv_k\ \bigg)$,  $\bar{\bM}_k = \bigg(\bv_1, \ \bv_2,\  \cdots, \ \bv_{i-1},\ \bv_{j+k},\ \bv_{i+1},\ \cdots,\ \bv_k\bigg)$ and $\bM_{k-1} = \bigg(\bv_1, \ \bv_2,\  \cdots, \ \bv_{i-1},\ \bv_{i+1},\ \cdots,\ \bv_k\bigg).$ 
  
Using Lemma \ref{lma_volume}, we have
\be
\lab \frac{\operatorname{det}(\bar{\bR}_{11}^{i,j})}{\operatorname{det}(\bR_{11})} \rab = \frac{V(\bar{\bM}_k)}{V(\bM_k)}
	=\frac{\min\limits_{x\in \R^{k-1}} \| \bM_{k-1}\bx - \bv_{j+k}\|_2}{\min\limits_{x\in \R^{k-1}} \| \bM_{k-1}\bx - \bv_{i}\|_2}.
\ee
Let $\bM^{sk}_k = \mmO\bM_k,\ \bar{\bM}^{sk}_k = \mmO\bar{\bM}_k$ and $\bM^{sk}_{k-1} = \mmO \bM_{k-1}$. Using Lemma \ref{lma_volume} again, we obtain 
\be
\lab \frac{\operatorname{det}(\bar{\bR}^{sk (i,j)}_{11})}{\operatorname{det}(\bR^{sk}_{11})} \rab = \frac{V(\bar{\bM}^{sk}_k)}{V(\bM^{sk}_k)} =\frac{\min\limits_{x\in \R^{k-1}} \| \bM^{sk}_{k-1}\bx - \bv^{sk}_{j+k}\|_2}{\min\limits_{x\in \R^{k-1}} \| \bM^{sk}_{k-1}\bx - \bv^{sk}_{i}\|_2}= \frac{\min\limits_{x\in \R^{k-1}} \| \mmO (\bM_{k-1}\bx - \bv_{k+j})\|_2}{\min\limits_{x\in \R^{k-1}} \| \mmO (\bM_{k-1}\bx - \bv_{i} )\|_2}.
\ee

The proof is finished using Lemma \ref{lma1}. 
\end{proof}

\section{Randomized strong rank-revealing QR}
\label{section_rrqr}
In this section, we introduce a randomized strong rank-revealing QR factorization and show that it provides rank-revealing guarantees comparable to those of the deterministic algorithm. We consider two scenarios, when the rank $k$ is specified and when a fixed tolerance $\tau$ is given. In the latter case, the factorization stops when the $\ell_2$-norm of the columns of $\bR_{22}$ is smaller than $\sqrt{\frac{1+\me}{1-\me}} \tau$.  We first consider the case when the dimension of the range of $\bM$, with $\bM \in \bR^{m \times n}$, is much smaller than $m$, as for example in the case of tall and skinny matrices or low-rank matrices, which arise often in practice \cite{big_data_low_rank}. Then we extend our analysis to general matrices for the case when $k$ is given and $k \ll m$.
  
\subsection{Algorithms}

We first consider the case when the dimension of the range of $\bM$ is much smaller than $m$. Given $k$, the goal is to select $k$ columns from $\bM$ which are then used to compute a partial pivoted QR factorization that satisfies the strong rank-revealing properties stated in \eqref{def_srrqr}.  The main idea of the randomized approach is to select the $k$ columns by performing a deterministic strong RRQR on its sketch $\bM^{sk} = \mmO \bM$.  We will show later that this selection ensures the preservation of strong rank-revealing properties, up to a factor of $\sqrt{\frac{1+\varepsilon}{1-\varepsilon}}$.  

\begin{algorithm}[htbp]
\caption{Randomized SRRQR (randSRRQR-$r$) for a given rank $k$}\label{alg:RSRRQR}
\begin{algorithmic}[1]
\Require $\bM$ is $m \times n$ matrix, parameter $f>1$ for strong RRQR, $d$ is sampling size  such that $d \ll m$ and $\mmO \in \bR^{d \times m}$ is $(\me,\delta,r)$-OSE, where $\operatorname{dim}(range(M))\leq r$,  $k \leq \min(n,d)$ is target rank.
\Ensure $\bQ$ is $m \times m$ orthogonal matrix, $\bR$ is $m \times n$ upper trapezoidal matrix, $\mmP$ is $n \times n$ permutation matrix such that $\bM\mmP = \bQ \bR$.
\Function{$[\bQ,\bR,\mmP] =$ randSRRQR-$r$}{$\bM, f, k , d$}
\State Generate a $d \times m$ sketching matrix $\mmO$
\State Compute $\bM^{sk} \coloneqq \mmO \bM$
\State Get permutation matrix $\mmP$ by performing SRRQR on $\bM^{sk}$ with given $k$ and $f$,
\Statex \hspace{1cm}
  [$\sim, \ \sim,\ \sim,\ \mmP]$ = SRRQR($\bM^{sk},f,k$) 
\State Do partial QR factorization on $\bM\mmP = \bQ \bR = \bQ \begin{pmatrix}
    \bR_{11} & \bR_{12} \\
    0        & \bR_{22}
\end{pmatrix}$ by a stable method
\State \Return $\bQ, \bR, \mmP$
\EndFunction
		\end{algorithmic}
	\end{algorithm}

The randomized strong RRQR factorization for a given rank $k$ is presented in Algorithm \ref{alg:RSRRQR}. Since we have no prior knowledge on $\bM$, we first generate an $(\me,\md,r)$-OSE sketching matrix $\mmO \in \bR^{d \times m}$ that satisfies the $\varepsilon$-embedding property for the range of $\bM$ with high probability, where $r$ is the dimension of the range of $\bM$.  As discussed in Section \ref{subsec:randomsketch}, $\mmO$ can be, for example, a rescaled Gaussian matrix or a subsampled randomized Hadamard transform. The sketch of the matrix is computed as $\bM^{sk} = \mmO \bM$.  A deterministic SRRQR is then performed on $\bM^{sk}$ that allows selecting $k$ columns and providing strong rank-revealing guarantees as in \eqref{def_srrqr} for the factorization of $\bM^{sk}$.  It returns a corresponding permutation matrix $\mmP$ that permutes the selected columns in leading positions. A partial QR factorization on the permuted matrix $\bM\mmP$ is finally performed. Efficient and numerically stable methods can be applied for this step, for instance TSQR \cite{cotsqrlu} or randomized Householder QR \cite{randhsqr}, which leverage parallelization and randomization, respectively.  One could also use $\bR_{11}^{sk}$ as a preconditioner for CholeskyQR, as e.g. explained in \cite{murray2023randomizednumericallinearalgebra}.

We note that the strong RRQR factorization of $\bM_{sk}$ is performed by calling Algorithm \ref{alg:SRRQR_efficient} that ensures that the strong rank-revealing property is progressively satisfied as $k$ increases. However we note that, as detailed in \cite{rqrcp}, the strong RRQR factorization can also be computed by first computing a QRCP factorization up to step $k$ and then performing additional permutations to satisfy the strong rank revealing property in \eqref{def_srrqr}. Both versions of the algorithms can be used in our work, and in particular the strong rank-revealing properties that we prove in Section \ref{subsec_randstrongproperties} hold.

The algorithm can be further extended to also determine an appropriate target rank $k$. One approach involves computing the singular values of $\bM^{sk}$, denoted as $\ms_i(\bM^{sk})$, and analyzing their decay. By Theorem \ref{th_sk_svd}, these singular values approximate well those of $\bM$, allowing us to identify a gap in the spectrum.

Algorithm \ref{alg:RSRRQR_k} presents the randomized strong RRQR when the rank $k$ is unknown. In this case the goal is to provide a factorization that ensures that the norm of each column of $\bR_{22}$ is bounded by a tolerance.  The algorithm is similar to Algorithm \ref{alg:RSRRQR}. However, in this case, given a tolerance $\tau$, the deterministic strong RRQR Algorithm \ref{alg:SRRQR_efficient} is called with the input tolerance $\tau$ and thus returns a parameter $k$ and a selection of columns $k$ such that the $\ell_2$-norm of each column of $\bR_{22}^{sk}$ is smaller than $\tau$. We will show in Theorem \ref{th_trailingmatrix} that this ensures that the $\ell_2$-norm of each column of $\bR_{22}$ is bounded by $\frac{\tau}{ \sqrt{1 - \me}}$.  This allows us to bound $\|\bR_{22}\|_F$. 

\begin{algorithm}[htbp]
\caption{Randomized SRRQR (randSRRQR-$\tau$) for a tolerance $\tau$}\label{alg:RSRRQR_k}
\begin{algorithmic}[1]
\Require 
		$\bM$ is $m \times n$ matrix, parameter $f>1$ for strong RRQR, $d$ is sampling size such that $d \ll m$ and $\mmO \in \bR^{d \times m}$ is $(\me,\delta,r)$-OSE, where $\operatorname{dim}(range(M))\leq r$, $\tau$ is tolerance parameter such that $\max\limits_{j} \mr_j (\bR_{22})\leq \frac{\tau}{ \sqrt{1 - \me}}$.
  \Ensure
        $\bQ$ is $m \times m$ orthogonal matrix, $\bR$ is $m \times n$ upper trapezoidal matrix, $\mmP$ is $n \times n$ permutation matrix such that $\bM\mmP = \bQ \bR$, $k$ is obtained rank.
\Function{$[k, \bQ,\bR,\mmP] =$ RandSRRQR-$\tau$}{$\bM, f, \tau ,d$}
\State Generate a $d \times m$ sketching matrix $\mmO$
\State Compute $\bM^{sk} \coloneqq \mmO \bM$
\State Get permutation matrix $\mmP$ and $k$ by performing SRRQR on $\bM^{sk}$ with given $f$ and $\tau$,
\Statex \hspace{1cm}
  [$k,\ \sim,\ \sim,\ \mmP]$ = SRRQR($\bM^{sk},f,\tau$) 
\State Do partial QR factorization on $\bM\mmP = \bQ \bR = \bQ \begin{pmatrix}
    \bR_{11} & \bR_{12} \\
    0        & \bR_{22}
\end{pmatrix}$ by a stable method
\Statex \quad \quad ($\bR_{11}$ is a $k \times k$ submatrix)
\State \Return $k, \bQ, \bR, \mmP$
\EndFunction
		\end{algorithmic}
	\end{algorithm}

For both algorithms, since $\bM^{sk}$ is a matrix of much smaller dimensions, computing its strong RRQR factorization is significantly more efficient than computing the strong RRQR factorization of the matrix $\bM$. In terms of computational cost, it is shown in \cite{srrqr} that the deterministic strong RRQR Algorithm \ref{alg:SRRQR_efficient} requires at most $2mk(2n-k) + 4t_fn(m+n)$ flops, where $t_f$ is the number of column interchanges needed to ensure that $\rho(\bR,k) \leq f$ and thus the strong rank-revealing properties are satisfied. It is also shown that $t_f$ is bounded by $k\log_f\sqrt{n}$. In our tests, in general $t_f$ is much smaller than this upper bound. In the randomized algorithms, since $M^{sk}$ is $d\times n$, the cost of this step is $2dk(2n-k) + 4t_fn(d+n)$ flops. The cost of sketching is $O(mnlog(m))$ flops. The final QR factorization without pivoting costs $4mnk$. The total cost of Algorithms \ref{alg:RSRRQR} and \ref{alg:RSRRQR_k} is thus at most $O(mnlog(m)) + 2dk(2n-k) + 4t_fn(d+n) + 4mnk$ flops. More importantly, if the algorithm is implemented on a parallel machine, the communication among processors is drastically reduced since the strong RRQR factorization of $\bM^{sk}$ can be performed on one processor and the QR factorization without permutations of $\bM$ can be computed using TSQR \cite{cotsqrlu} for example.  
% If $n < d \ll m$, the second term can be neglected.
%In the case of Algorithm \ref{alg:RSRRQR_k}, if the algorithm halts when $k = k_f$ and $t_f$ is the number of column interchanges, the total cost is at most $O(mlog(m)) + 2dk(2n-k) + 4t_fn(d+n) + \frac{2mn^2}{P} + \frac{2n^3}{3} log(P)$ flops. 

We now discuss the case of a general matrix $\bM \in \R^{m \times n}$ $(m \geq n)$ for which we want to compute a strong RRQR factorization for a given $k$, where $k \ll m$. The randomized strong RRQR factorization presented in Algorithms \ref{alg:RSRRQR} and \ref{alg:RSRRQR_k} relies on a random sketching matrix that is an $(\me, \delta, r)$-OSE, where $r$ is the dimension of the range of $\bM$ and thus embeds the subspace $range(\bM)$ with high probability. For a full rank matrix, the embedding dimension $d$ is larger than $n$. 
%For instance, if we use an SRHT with distortion parameter $\me = \frac{1}{2}$, $d$ can be set as $(1+c)n \log(n)$ \cite{tropp_improved_2011}, where $c$ is a sufficiently small constant. 
Hence, when $n$ is close to or equal to $m$, this embedding dimension can become close to, or even larger than $m$, making dimension reduction ineffective.  Thus, this randomized approach cannot be applied to compute a complete factorization, it can only be used to compute a partial factorization of dimension $k$, for $k \ll m$.  Such a partial factorization has many applications. It can be used to select $k$ columns of a matrix or compute its rank-k approximation. 
%For a matrix $\bM$ with numerical rank $k$, we know that the dimension of $range(\bM)$ is close to k. 
In this case, we use an $(\me, \delta, k+1)$-OSE $\mmO$ to embed any $k+1$ dimensional subspace with high probability.
%$range(\bM)$ and some perturbations of $range(\bM)$. 
The modified version of Algorithm \ref{alg:RSRRQR} is presented as Algorithm \ref{alg:RSRRQR_square}. The difference is in line 2, where we generate a $d \times m$ random sketching matrix $\mmO$, which is an $(\me,\delta,k+1)$-OSE, instead of an $(\me,\delta,r)$-OSE. 
%The parameter $d$ can be set as $O(k log(m)/log(k)$.

\begin{algorithm}[htbp]
\caption{Randomized SRRQR(randSRRQR-$r$) for a target rank $k$}\label{alg:RSRRQR_square}
\begin{algorithmic}[1]
\Require $\bM \in \bR^{m \times n}$, target rank $k \ll m$, parameter $f>1$, sampling size $d$ such that $\mmO$ is $(\me,\delta,k+1)$-OSE.
\Ensure $\bQ$ is $m \times m$ orthogonal matrix, $\bR$ is $m \times n$ upper trapezoidal matrix, $\mmP$ is $n \times n$ permutation matrix such that $\bM\mmP = \bQ \bR$.
\Function{$[\bQ,\bR,\mmP] =$ RandSRRQR-$r$}{$\bM, f, k, d$}
\State Generate an $(\me,\delta,k+1)$-OSE matrix $\mmO \in \bR^{d \times m}$
\State Compute $\bM^{sk} = \mmO \bM$
\State Get permutation matrix $\mmP$ by performing SRRQR on $\bM^{sk}$ with given $k$ and $f$,
\Statex \hspace{1cm}
  [$\sim,\ \sim,\ \sim,\ \mmP]$ = SRRQR($\bM^{sk},f,k$) 
\State Compute partial QR factorization of $\bM\mmP = \bQ \bR$ by a stable method
\State \Return $\bQ, \bR, \mmP$
\EndFunction
\end{algorithmic}
	\end{algorithm}

\subsection{Strong rank-revealing properties}\label{srrqr_pp}
\label{subsec_randstrongproperties}
We present now our main theoretical results. Specifically, we demonstrate that the randomized strong rank-revealing factorizations presented in Algorithms \ref{alg:RSRRQR} and \ref{alg:RSRRQR_k} nearly preserve the strong rank-revealing properties from \eqref{def_srrqr}. This shows that the permutation matrix $\mmP$ obtained from the strong RRQR of $\bM^{sk}$ provides an effective column selection for $\bM$. To prove this result we use Lemma \ref{lma2} from the previous section showing that the ratio of the volumes of two matrices that differ by only one column is nearly preserved by sketching.  Furthermore, the tolerance parameter $\tau$ in Algorithm \ref{alg:RSRRQR_k} allows us to control the $L_2$-norms of the columns of both $\bR_{22}$ and $\bR_{22}^{sk}$. One of our key findings, introduced below, shows how the critical quantity $\rho(\bR,k)$ in Algorithm \ref{alg:SRRQR_efficient}, defined in \eqref{def_rho}, is influenced by the sketching matrix. 

\begin{theorem} \label{th_rsrrqr1}
Consider $\bM \in \bR^{m \times n}$ and assume that $\mmO\in\Rdm$ is an $\me$-embedding of $range(\bM)$. Let $\bM^{sk} = \mmO\bM$ and let $\bM^{sk} \mmP = \bQ^{sk} \bR^{sk}$ be the strong RRQR factorization obtained for a given $k$ and $f$ (or for a given $\tau$ and $f$).  Consider the partial QR factorization of $\bM\  \mmP$ and let $\mathcal{R}_k(\bM\  \mmP) = \begin{pmatrix}
\bR_{11}      & \bR_{12}\\
\bO    & \bR_{22}
\end{pmatrix}$, where $\bR_{11} \in \bR^{k \times k}$. We have 
		\be
         1 \leq \frac{\ms_i(\bM)}{\ms_i(\bR_{11})}, \frac{\ms_j(\bR_{22})}{\ms_j(\bM)} \leq \sqrt{1+\tilde{f}^2k(n-k)}, \ \lab (\bR_{11}^{-1}\bR_{12})_{i,j} \rab \leq \tilde{f},
		\ee
	for all $\ 1\leq i \leq k,\ 1\leq j\leq n-k$, where $\tilde{f} = \sqrt{\frac{1+\me}{1-\me}}f$. 
	\end{theorem}
	\begin{proof}
The lower bound is given by the interlacing property of singular values \cite{golub2013matrix}. The strong RRQR factorization of $\bM^{sk}$ interchanges columns until there is no pair $(i,j)$ such that $\lab \operatorname{det}(\bar{\bR}_{11}^{sk (i,j)})/\operatorname{det}(\bR_{11}^{sk}) \rab \geq f$. Hence, it is guaranteed that
\be
\forall 1\leq i \leq k,\ 1\leq j\leq n-k, \ \lab \operatorname{det}(\bar{\bR}_{11}^{sk (i,j)})/\operatorname{det}(\bR_{11}^{sk}) \rab < f.
\ee
Since $\mmO$ is an $\me$-embedding of $range(\bM)$, using Lemma \ref{lma2}, we have
        \be
        \label{ieq_alldet}
		\forall 1\leq i \leq k,\ 1\leq j\leq n-k, \ \lab \operatorname{det}(\bar{\bR}_{11}^{(i,j)})/\operatorname{det}(\bR_{11}) \rab < \sqrt{\frac{1+\me}{1-\me}} f.
		\ee
We obtain that $\rho(\bR,k) = \max\limits_{\ijk} \lab \operatorname{det}(\bar{\bR}_{11}^{(i,j)})/\operatorname{det}(\bR_{11}) \rab < \sqrt{\frac{1+\me}{1-\me}}f = \tilde{f}$.  Using Theorem \ref{th_srrqr1}, we conclude that
		$$
		1 \leq \frac{\ms_i(\bM)}{\ms_i(\bR_{11})}, \frac{\ms_j(\bR_{22})}{\ms_j(\bM)} \leq \sqrt{1+\tilde{f}^2k(n-k)}, \ \lab (\bR_{11}^{-1}\bR_{12})_{i,j} \rab \leq \tilde{f}.
		$$
	\end{proof}

In this proof, \eqref{ieq_alldet} is essentially derived from Lemma \ref{lma1}. If $\mmO$ is an $\me$-embedding of $range(\bM)$, then $\mmO$ is an $\me$-embedding of $span\{\bv_1,\ \bv_2,\ \bv_3,\cdots,\bv_k\}$ for any columns $\bv_1,\ \bv_2,\ \bv_3,\cdots,\bv_k$ of $\bM$. Hence the residuals of all least square problems of the form $\min\limits_{\bx \in \R^{k-1}}\|\bM_{k-1} \bx - \bv_i\|_2$ can be nearly preserved, where $\bM_{k-1} = \bigg(\bv_1, \ \bv_2,\  \cdots, \ \bv_{i-1},\ \bv_{i+1},\ \cdots,\ \bv_k\bigg)$. 
 
The result of Theorem \ref{th_rsrrqr1} holds under two conditions: $\mmO$ is an $\me$-embedding of $range(\bM)$ and $\rho(\bR^{sk},k) \leq f$. Algorithm \ref{alg:RSRRQR} starts by generating an $(\me,\md,r)$-OSE $\mmO$ so that $\mmO$ is an $\me$-embedding of $range(\bM)$ with probability at least $1 - \md$, provided that the dimension of the range of $\bM$ is smaller than $r$. Therefore, the factorization computed by Algorithm \ref{alg:RSRRQR} satisfies the strong rank-revealing property in Theorem \ref{th_rsrrqr1} with probability at least $1 - \md$.  Typically, $\me$ is taken to be $\frac{1}{4}$ and $f$ is set as $2$, then $\tilde{f} = \sqrt{\frac{1+\me}{1-\me}}f$ is equal to $2.58$.  Hence the guarantees of the randomized strong RRQR are very similar to those of the deterministic strong RRQR. Another result that can be derived from Lemma \ref{lma2} is that the Frobenius norm of $\bR_{22}$ can be as small as that of $\bR^{sk}_{22}$. We prove first a result that bounds the $\ell_2$-norms of the columns of $\bR_{22}$. 
\begin{theorem}\label{th_trailingmatrix}
Consider $\bM \in \bR^{m \times n}$ and assume that $\mmO\in\Rdm$ is an $\me$-embedding of $range(\bM)$.  Let $\bM^{sk} = \mmO\bM$ and for any permutation matrix $\mmP$, we consider the partial QR factorization of $\bM\  \mmP$ and $\bM^{sk}\  \mmP$: $\cR_k(\bM \mmP) = \begin{pmatrix}
			\bR_{11}       & \bR_{12}\\
			\bO    & \bR_{22}
   \end{pmatrix}$ and $\cR_k(\bM^{sk} \mmP) = \begin{pmatrix}
			\bR^{sk}_{11}       & \bR^{sk}_{12}\\
			\bO    & \bR^{sk}_{22}
   \end{pmatrix}$, where $\bR_{11},\bR_{11}^{sk}\in \R^{k\times k}$.  We have 
   \be
   \frac{1}{1 + \me} \mr_j^2(\bR_{22}^{sk})  \leq \mr_j^2(\bR_{22})  \leq \frac{1}{1 - \me}\mr_j^2(\bR_{22}^{sk}),\ \forall 1\leq j\leq n - k.
   \ee
 \end{theorem}
 \begin{proof}
     Let $\bM\mmP = \bigg(\bv_1, \ \bv_2\ ,\  \cdots \ , \bv_n\  \bigg)$ and let $\bM_k = \bigg(\bv_1 \ , \ \bv_2\ ,\  \cdots , \ \bv_k\ \bigg)$ be the matrix formed by the first $k$ columns of $\bM\mmP$. We prove the upper bound, and the lower bound can be obtained by a similar reasoning. Using Lemma \ref{lma1}, for all $1\leq j\leq n - k$, we have
     \[
    \begin{aligned}
        \mr_j^2(\bR_{22}) & = \min\limits_{x \in \R^k} \|\bM_k \bx - v_{j+k}\|_2^2 \\
        &\leq \frac{1}{1 - \me}  \min\limits_{x \in \R^k} \|\mmO(\bM_k \bx - v_{j+k})\|_2^2\\
        &= \frac{1}{1 - \me}\mr_j^2(\bR_{22}^{sk}).
    \end{aligned}
     \]
 \end{proof}
 \begin{corollary}
     Under the assumptions of Theorem \ref{th_trailingmatrix}, we have 
     \be
   \frac{1}{1 + \me}\|\bR^{sk}_{22}\|^2_F \leq \|\bR_{22}\|^2_F \leq \frac{1}{1 - \me}\|\bR^{sk}_{22}\|^2_F.
   \ee
 \end{corollary}
 \begin{proof}
 Using Theorem \ref{th_trailingmatrix}, we have
 \[
      \|\bR_{22}\|_F^2 = \sum_{j = 1}^{n - k} \mr_j^2(\bR_{22})
      \leq \frac{1}{1 - \me}\sum_{j = 1}^{n - k}  \mr_j^2(\bR_{22}^{sk})
      = \frac{1}{1 - \me} \|\bR_{22}^{sk}\|_F^2.
 \]
 The bound $\|\bR_{22}\|_F^2 \geq \frac{1}{1 + \me} \|\bR_{22}^{sk}\|_F^2$ can be obtained with a similar reasoning.
 \end{proof}

The step 4 of Algorithm \ref{alg:RSRRQR_k} guarantees that the trailing submatrix in $\cR_k(\bM^{sk}\mmP) = \begin{pmatrix}
    \bR^{sk}_{11}       & \bR^{sk}_{12}\\
			\bO    & \bR^{sk}_{22}
\end{pmatrix}$ satisfies 
$$
\mr_j(\bR^{sk}_{22}) \leq \tau,\ \forall 1\leq j \leq n-k.
$$
Using Theorem \ref{th_trailingmatrix}, we conclude that when a threshold is given, we have
$$
\mr_j(\bR_{22}) \leq \frac \tau {\sqrt{1 - \me}},\ \forall 1\leq j \leq n-k.
$$

In the following theorem, we show that Algorithm \ref{alg:RSRRQR_square} computes a factorization that satisfies the strong rank-revealing property from \eqref{def_srrqr}. Unlike Theorem \ref{th_rsrrqr1}, which was proven for tall and skinny matrices or low-rank matrices, we assume here that $\mmO$ is an $(\me,\md,k+1)$-OSE rather than an $\me$-embedding. The key distinction is that an $\me$-embedding guarantees the embedding of one specific subspace, whereas an OSE ensures that any $(k+1)$-dimensional subspace is embedded with high probability. In the case of general matrices, several subspaces, spanned by a subset of the columns of $\bM$, need to be embedded. Since we do not have prior knowledge of this subspace before computing the strong RRQR factorization of $\bM^{sk}$, the use of an OSE is necessary.

\begin{theorem} \label{th_rsrrqr_square}
    Given a matrix $\bM \in \Rmn$ and an $(\me,\md,k+1)$-OSE $\mmO$, let $\bM^{sk} = \mmO\bM$.  Let $\bM^{sk} \mmP = \bQ^{sk} \bR^{sk} = \bQ^{sk}\begin{pmatrix}
			\bR^{sk}_{11}      & \bR^{sk}_{12}\\
			\bO    & \bR^{sk}_{22}
		\end{pmatrix}$ be the strong RRQR factorization of $\bM^{sk}$ with given $k$ and $f$.  Let $\bM \mmP = \bQ \bR$ be the partial QR with no pivoting factorization of $\bM\ \mmP$ and  $\mathcal{R}_k(\bM\  \mmP) = \begin{pmatrix}
			\bR_{11}      & \bR_{12}\\
			\bO    & \bR_{22}
		\end{pmatrix}$.
 With probability at least $1 - \md$, the QR factorization $\bM \mmP = \bQ \bR$ satisfies 
    \be \label{RRQR} 
         1 \leq \frac{\ms_i(\bM)}{\ms_i(\bR_{11})}, \frac{\ms_j(\bR_{22})}{\ms_{j + k}(\bM)} \leq \sqrt{1+\tilde{f}^2k(n-k)}, \ |(\bR_{11}^{-1}\bR_{12})_{i,j}| \leq \tilde{f}
		\ee
		for all $1\leq i \leq k,\ 1\leq j\leq n-k$, where $\tilde{f} = \sqrt{\frac{1+\me}{1-\me}}f$.
\end{theorem}
\begin{proof}
    We have the following two factorizations, $\bM \mmP = \bQ \bR$ and $\bM^{sk} \mmP = \bQ^{sk} \bR^{sk}$, where $\bR = \mathcal{R}_k(\bM \mmP)= \begin{pmatrix}
			\bR_{11}      & \bR_{12}\\
			\bO    & \bR_{22}
		\end{pmatrix}$ and $\bR^{sk} = \mathcal{R}_k(\bM^{sk} \mmP) =\begin{pmatrix}
			\bR^{sk}_{11}      & \bR^{sk}_{12}\\
			\bO    & \bR^{sk}_{22}
		\end{pmatrix}.$ 

We consider increasing the value of the determinant of $\bR_{11}$ by interchanging two columns, and we suppose that the ratio  $\operatorname{det}(\bar{\bR}_{11}^{(i,j)})/\operatorname{det}(\bR_{11})$ is maximized when the columns $i_0$ and $j_0+k$ are interchanged, where $\bar{\bR}_{11}$ is the factor obtained after interchanging $i_0$ and $j_0+k
$, $i.e.\  (i_0, j_0) = \argmax\limits_{\ijk} \lab \operatorname{det}(\bar{\bR}_{11})/\operatorname{det}(\bR_{11}) \rab $. 

By the strong RRQR of $\bM^{sk}$, the step in line 4 of Algorithm \ref{alg:RSRRQR_square} guarantees that $\rho(\bR^{sk},k) = \max\limits_{\ijk} \lab \operatorname{det}(\bar{\bR}^{sk  (i,j}_{11})/\operatorname{det}(\bR^{sk)}_{11}) \rab \leq f$.  

Let $\bM\mmP = \bigg(\bv_1, \ \bv_2,\ \cdots, \ \bv_n\ \bigg)$ and $\bM^{sk}\mmP = \bigg(\bv_1^{sk}, \ \bv_2^{sk},\  \cdots,\ \bv_n^{sk}\bigg)$, where $\bv_i^{sk} = \mmO \bv_i, \ \forall 1\leq i \leq n$. Let $\bM_k = \bigg(\bv_1, \ \bv_2,\  \cdots, \ \bv_k\ \bigg)$,  $\bar{\bM}_k = \bigg(\bv_1, \ \bv_2,\  \cdots, \ \bv_{i_0-1},\ \bv_{j_0+k},\ \bv_{i_0+1},\ \cdots,\ \bv_k\bigg)$ and $\bM_{k-1} = \bigg(\bv_1, \ \bv_2,\  \cdots, \ \bv_{i_0-1},\ \bv_{i_0+1},\ \cdots,\ \bv_k\bigg).$ In addition, let $\cV_1$ denote $range(\bM_k)$ and $\cV_2$ denote $range(\bar{\bM}_k)$.
  
Note that $dim(span(\cV_1 \cup \cV_2)) = k+1$, the $(\me,\delta,k+1)$-OSE matrix $\mmO$ is an $\me$-embedding of $span(\cV_1 \cup \cV_2)$ with probability at least $1-\md$. Then by Lemma \ref{lma1}, given that $\mmO$ is an $\me$-embedding of $span(\cV_1 \cup \cV_2)$, the following two inequalities hold with probability $1 - \delta$:
 \be \label{event1}
		\frac{1}{\sqrt{1+\me}}\min\limits_{x\in \R^{k-1}}\|\mmO ( \bM_{k-1}\bx-\bv_{i_0} )\|_2 \leq \min\limits_{x\in \R^{k-1}}\| \bM_{k-1}\bx-\bv_{i_0}\|_2 \leq \frac{1}{\sqrt{1-\me}}\min\limits_{x\in \R^{k-1}}\|\mmO ( \bM_{k-1}\bx-\bv_{i_0} )\|_2,
		\ee 
  \be \label{event2}
		\frac{1}{\sqrt{1+\me}}\min\limits_{x\in \R^{k-1}}\|\mmO ( \bM_{k-1}\bx-\bv_{j_0+k} )\|_2 \leq \min\limits_{x\in \R^{k-1}}\| \bM_{k-1}\bx-\bv_{j_0+k}\|_2 \leq \frac{1}{\sqrt{1-\me}}\min\limits_{x\in \R^{k-1}}\|\mmO ( \bM_{k-1}\bx-\bv_{j_0+k} )\|_2.
		\ee
 Deriving from \eqref{event1} and \eqref{event2}, we obtain 
  \[
  \begin{aligned}
      \rho(\bR,k) = \max\limits_{1\leq i\leq k,1\leq j\leq n-k} \lab \frac{\operatorname{det}(\bar{\bR}_{11}^{(i,j)})}{\operatorname{det}(\bR_{11})} \rab & =\frac{\min\limits_{x\in \R^{k-1}}\| \bM_{k-1}\bx-\bv_{j_0+k}\|_2}{\min\limits_{x\in \R^{k-1}}\|\bM_{k-1}\bx-\bv_{i_0}\|_2}\\
      &\leq\sqrt{\frac{1+\me}{1-\me}} \frac{\min\limits_{x\in \R^{k-1}}\|\mmO ( \bM_{k-1}\bx-\bv_{j_0+k} )\|_2}{\min\limits_{x\in \R^{k-1}}\|\mmO ( \bM_{k-1}\bx-\bv_{i_0} )\|_2}\\
      &=\sqrt{\frac{1+\me}{1-\me}} \lab \operatorname{det}(\bar{\bR}_{11}^{sk (i,j)})/\operatorname{det}(\bR_{11}^{sk}) \rab\\
      &\leq\sqrt{\frac{1+\me}{1-\me}}\rho(\bR^{sk},k)\\
      &\leq \sqrt{\frac{1+\me}{1-\me}}f.
  \end{aligned}
  \]
  
  The proof is finished by using Theorem \ref{th_srrqr1}.
\end{proof}

We conclude that randomized strong rank-revealing QR can be used when the target rank $k$ is small, even if the number of columns $n$ is large and the matrix is full rank. The algorithms presented can also be extended to block algorithms, using the same approach as randomized QRCP. We do not discuss this further in this paper, since the extension is straightforward. 

\section{Numerical experiments}
\label{section_tests}
In this section we analyze the numerical efficiency of the randomized strong RRQR factorization presented in Algorithms \ref{alg:RSRRQR} and \ref{alg:RSRRQR_k}. We focus in particular on the approximation of the singular values for different test matrices by considering the ratios $\ms_i(\bM)/\ms_i(\bR_{11})$ and $\ms_j(\bR_{22})/\ms_{j+k}(\bM)$. Since in practice QR with column pivoting is widely used and provides good approximations of the singular values, we compare the results of randSRRQR-$\tau$ presented in Algorithm \ref{alg:RSRRQR}, or randSRRQR-$r$ presented in Algorithm \ref{alg:RSRRQR_k} with those of the deterministic strong RRQR and QRCP. Furthermore, we also consider the diagonal entries of the $\bR$ factors computed by randSRRQR-$\tau$ (or randSRRQR-$r$) (referred to as $R$-values) and those of the $\textbf{L}$ factors (referred to as $L$-values) obtained from the QLP factorization \cite{qlp}. This factorization is obtained as follows. First, the randSRRQR-$\tau$ or randSRRQR-$r$ factorization of $\bM$ is computed to obtain $\bM \mmP = \bQ \bR$. Then the QR factorization with no pivoting of $\bR^T$ leads to $\bR^T = \bP \bL^T $, where $\bL$ is lower-triangular. The QLP factorization is thus $\bM\mmP = \bQ\bL\bP^T$. In practice, these values are observed to provide also an approximation of the singular values of $\bM$. The experiments are performed using Matlab on a laptop with a 3.4 GHz Intel Core i5 CPU and 16GB of RAM. 

We use the following test matrices of dimensions $m \times n$:
\begin{enumerate}
    \item 
    Kahan matrix: an $n\times n$ Kahan matrix defined as
\be \label{def:kahan}
\bK_n = diag(1, s, s^2, \cdots, s^{n-1})\begin{pmatrix}
	1 & -c & -c & \cdots & -c\\
	  0 & 1 & -c &  \cdots & -c\\
        0 & 0 & 1 &  \cdots & -c\\
        \vdots & \vdots & \vdots & \ddots & \vdots \\
        0 & 0 & 0 & \cdots & 1
	\end{pmatrix},
\ee
where $c>0$, $s>0$, and $c^2 + s^2 = 1$. To be able to use an SRHT as a random sketching matrix, the number of rows $m$ should be a power of $2$. Thus we extend $\bK_n$ with a block of zeros to obtain a tall and skinny matrix 
\be
\bM = \begin{pmatrix}
    \bK_n    \\
	\bO  
\end{pmatrix}
\ee
of dimensions $m\times n, m\gg n$, so that $m$ is a power of $2$.
\item Devil's stairs: a matrix with several gaps between its singular values. In the $k$th stair, all the singular values of $\bM$ are equal to $q^k$, where $q$ is a parameter and $0<q<1$. The codes from \cite{qlp} are used to generate this matrix, with $q = 10^{-3}$ and the length of the stairs equal to $100$.
\item Stewart matrix: matrix $\bM = \bU \pmb{\Sigma} \bV + 0.1 \ms_n * \text{rand}(m,n)$, see description in the text and \cite{qlp}.
\item H-C matrix: matrix with prescribed singular values, see description in the text and \cite{qlp_lra}.
\end{enumerate}

Three different dimensions are used in the experiments for each test matrix: $8192\times500$, $16384\times1000$, and $32768\times2000$. We use Algorithm \ref{alg:RSRRQR_k} for each test matrix. It computes a randomized strong RRQR for a given tolerance and uses as tolerance $\tau = 1e-10$ and $f = 2$ as parameters for the deterministic SRRQR of the sketch matrix. For Kahan matrix, we do an additional experiment, which is labeled as Kahan*, by using  Algorithm \ref{alg:RSRRQR} that computes a randomized strong RRQR for a given rank $k$ with parameters $f = 2$ and $k = n-1$.  The random sketching matrix $\mmO \in \bR^{d \times m}$ used in the experiments is an SRHT.  Unless otherwise specified, the sketch dimension $d$ is set as $\lfloor 3nlog(m)/log(n)\rfloor$.

Table \ref{table_time_rank} summarizes the test matrices and their dimensions. It also displays the numerical rank computed by strong RRQR from Algorithm \ref{alg:SRRQR_efficient} and randomized strong RRQR from Algorithm \ref{alg:RSRRQR_k} and the time (in seconds) required to compute it. These results show that the randomized version accelerates computing a strong rank-revealing factorization. It leads to a factor of $15.9\times$ speedup for the matrix Devil's stairs of dimensions $8192\times 500$. We observe that the numerical ranks computed by the two methods can be slightly different, since in the deterministic case the rank is determined by $\bM$, while in the randomized case it is determined by $\bM^{sk}$. The norm of each column of $\bM^{sk}$ can be slightly larger or smaller than that of $\bM$.  Consider the two strong rank-revealing QR factorizations obtained by Algorithm \ref{alg:SRRQR_efficient} and Algorithm \ref{alg:RSRRQR_k} as
$\bM \mmP = \bQ \bR = \bQ \begin{pmatrix}
    \bR_{11} & \bR_{12}\\
    \bO      & \bR_{22}
\end{pmatrix}$ and 
$\bM \mmP' = \bQ' \bR' = \bQ' \begin{pmatrix}
    \bR_{11}' & \bR_{12}'\\
    \bO      & \bR_{22}'
\end{pmatrix}$ respectively, where $\bR_{11}\in \R^{k_1 \times k_1}$ and $\bR_{11}'\in \R^{k_2 \times k_2}$.  We have that $k_1$ is the smallest integer such that $\max\limits_{j}\mr_j(\bR_{22}) \leq \tau$ under the constraint $\rho(\bR,k_1)\leq f$, and $k_2$ is the smallest integer such that $\max\limits_{j} \mr_j(\bR_{22}^{sk}) \leq \tau$ under the constraint $\rho(\bR^{sk},k_2)\leq f$, where $\bR^{sk} = \cR_k(\bM^{sk} \mmP') = \begin{pmatrix}
    \bR_{11}^{sk} & \bR_{12}^{sk}\\
    \bO      & \bR_{22}^{sk}
\end{pmatrix}$. Using Theorem and \ref{th_rsrrqr1} and \ref{th_trailingmatrix}, we have that $k_2$ satisfies $\max\limits_{j}\mr(\bR_{22}') \leq \frac \tau {\sqrt{1 - \me}}$ and $\rho(\bR',k_2)\leq \sqrt{\frac{1+\me}{1-\me}} f$.

\begin{table}[htbp]
\begin{tabular}{|l|l|ll|ll|}
\hline
\multirow{2}{*}{Matrix}          & \multirow{2}{*}{Dimensions} & \multicolumn{2}{l|}{Performing SRRQR on $\bM$}            & \multicolumn{2}{l|}{Performing SRRQR on $\bM^{sk}$}               \\ \cline{3-6} 
                                 &                       & \multicolumn{1}{l|}{Execution time(s)} & Rank & \multicolumn{1}{l|}{Execution time(s)} & Rank \\ \hline
\multirow{3}{*}{Kahan}           & $8192 \times 500$              & \multicolumn{1}{l|}{0.3267}            & 333  & \multicolumn{1}{l|}{0.0493}            & 337  \\ \cline{2-6} 
                                 & $16384 \times 1000$            & \multicolumn{1}{l|}{36.1454}           & 547  & \multicolumn{1}{l|}{5.7634}            & 528  \\ \cline{2-6} 
                                 & $32768 \times 2000$            & \multicolumn{1}{l|}{281.5560}          & 1547 & \multicolumn{1}{l|}{9.7826}            & 1520 \\ \hline
\multirow{3}{*}{Kahan*}          & $8192 \times 500$              & \multicolumn{1}{l|}{0.2411}            & 499  & \multicolumn{1}{l|}{0.0388}            & 499  \\ \cline{2-6} 
                                 & $16384 \times 1000$            & \multicolumn{1}{l|}{1.1168}            & 999  & \multicolumn{1}{l|}{0.3806}            & 999  \\ \cline{2-6} 
                                 & $32768 \times 2000$            & \multicolumn{1}{l|}{7.0915}            & 1999 & \multicolumn{1}{l|}{4.1509}            & 1999 \\ \hline
\multirow{3}{*}{Stewart}         & $8192 \times 500$              & \multicolumn{1}{l|}{0.7033}            & 99   & \multicolumn{1}{l|}{0.0359}            & 100   \\ \cline{2-6} 
                                 & $16384 \times 1000$            & \multicolumn{1}{l|}{5.1523}            & 97   & \multicolumn{1}{l|}{0.9206}            & 97   \\ \cline{2-6} 
                                 & $32768 \times 2000$            & \multicolumn{1}{l|}{54.8160}           & 95   & \multicolumn{1}{l|}{8.4531}            & 95   \\ \hline
\multirow{3}{*}{Deveil's stairs} & $8192 \times 500$              & \multicolumn{1}{l|}{0.5857}            & 400  & \multicolumn{1}{l|}{0.0369}            & 400  \\ \cline{2-6} 
                                 & $16384 \times 1000$            & \multicolumn{1}{l|}{6.7717}            & 800  & \multicolumn{1}{l|}{0.8664}            & 799  \\ \cline{2-6} 
                                 & $32768 \times 2000$            & \multicolumn{1}{l|}{61.1059}           & 1600 & \multicolumn{1}{l|}{8.3602}            & 1600 \\ \hline
\multirow{3}{*}{H-C}             & $8192 \times 500$              & \multicolumn{1}{l|}{0.5351}            & 334  & \multicolumn{1}{l|}{0.0778}            & 333  \\ \cline{2-6} 
                                 & $16384 \times 1000$            & \multicolumn{1}{l|}{4.6086}            & 667  & \multicolumn{1}{l|}{1.0987}            & 665  \\ \cline{2-6} 
                                 & $32768 \times 2000$            & \multicolumn{1}{l|}{36.3005}           & 1334 & \multicolumn{1}{l|}{13.5850}           & 1331 \\ \hline
\end{tabular}
\caption{Deterministic SRRQR versus Randomized SRRQR: Execution time.}
\label{table_time_rank}
\end{table}

We first discuss the Kahan matrix. Because of its special structure, no columns are interchanged during QRCP. Thus, the factorization computed by QRCP is exactly $\bM \bI = \bI \bM$, that is, the $\bQ$ factor is the identity matrix and the $\bR$ factor is the Kahan matrix itself. The results obtained by randomized SRRQR from Algorithm \ref{alg:RSRRQR} (randSRRQR-$r$) for a given rank $k = 499$ and $f = 2$ are presented in Figure \ref{fig:Kahan1}.  The dimension of the Kahan matrix is $8192 \times 500$. The ratios of the singular values $\ms_i(M)/\ms_i(R_{11})$ are displayed for both algorithms, QRCP and randSRRQR-$r$. We observe that the ratios are close to $1$ for QRCP, except for the last one, which becomes exponentially large. The ratios obtained by randSRRQR-$r$ are closer to $1$ and the last $10$ ratios have the first $10$ decimals $0$. The last ratios are also given in Table \ref{tab:Kahan}. The singular values of $\bM$ and those of $\bR_{11}$ computed by randSRRQR-$r$, SRRQR and QRCP are also given in Figure \ref{fig:Kahan2}, where we also display the obtained L-values and R-values.

% plots for Kahan matrix
\begin{figure}[htbp]
\begin{minipage}[t]{0.48\textwidth}
\centering
    \includegraphics[width = 8.23cm]{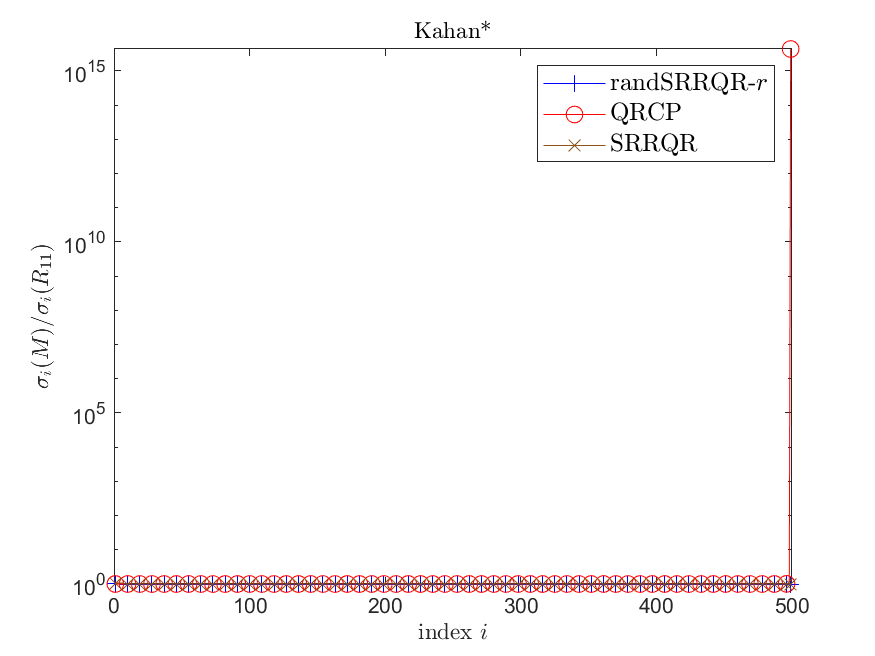}
    \caption{Ratios $\ms_i(\bM)/\ms_i(\bR_{11})$ of $8192\times500$ Kahan matrix}
    \label{fig:Kahan1}
\end{minipage}
    \begin{minipage}[t]{0.48\textwidth}
    \centering
    \includegraphics[width = 8.23cm]{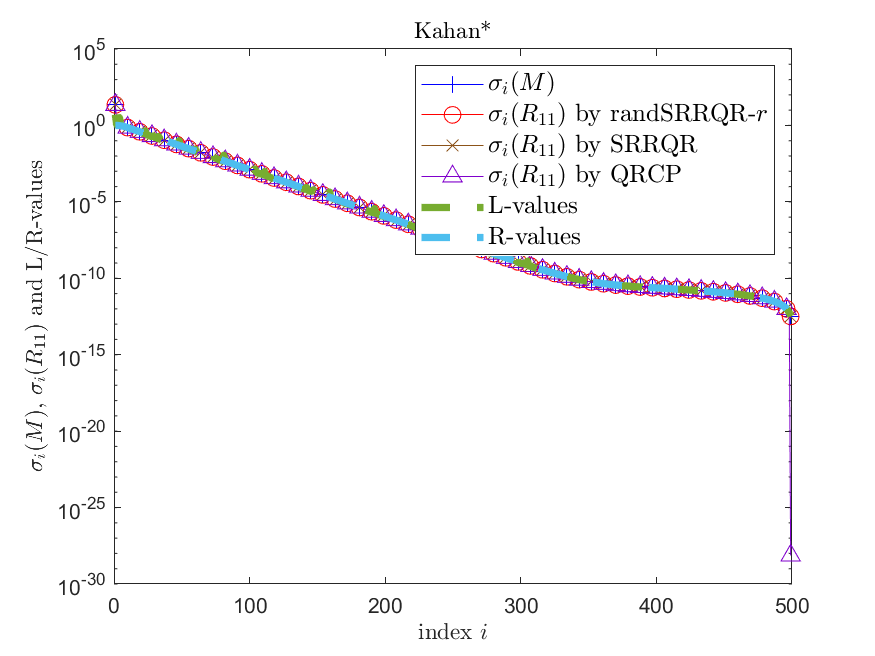}
    \caption{Singular values and L/R-values for $8192\times500$ Kahan matrix}
    \label{fig:Kahan2}
\end{minipage}
\end{figure}

% table of ratios for kahan and devil
 \begin{table}
    \begin{minipage}[c]{0.45\textwidth}
    \centering
    \begin{tabular}{|c|c|c|}
\hline
index j & randSRRQR-$r$ & QRCP \\
\hline
494   & 1.0000& 1.0050\\
\hline
495   & 1.0000& 1.0070\\
\hline
496   & 1.0000& 1.0107\\
\hline
497   & 1.0000& 1.0186\\
\hline
498   & 1.0000& 1.0413\\
\hline
499   & 1.0000& 4.3406e+15  \\\hline               
\end{tabular}
\caption{Ratios $\ms_i(M)/\ms_i(\bR_{11})$ for $8192\times500$ Kahan matrix}
    \label{tab:Kahan}
    \end{minipage}
    \begin{minipage}[c]{0.45\linewidth}
    \centering
    \begin{tabular}{|c|c|c|}
\hline
index j & randSRRQR-$\tau$ & QRCP \\
\hline
395   & 6.6429& 6.4632\\
\hline
396   & 6.9252& 7.0764\\
\hline
397   & 8.1693& 7.3105\\
\hline
398   & 9.0129& 8.8939\\
\hline
399   & 13.2641& 12.4902\\
\hline
400   & 15.8370 & 16.1388  \\\hline           
\end{tabular}
\caption{Ratios $\ms_i(M)/\ms_i(\bR_{11})$ for $8192\times500$ Devil's stairs}
    \label{tab:Devil1}
    \end{minipage}
\end{table}

The results obtained for the devil's stairs matrix of dimension $8192\times500$ are presented in Figures \ref{fig:Devil1} and \ref{fig:Devil2}. We observe that the ratios $\ms_i(M)/\ms_i(R_{11})$ obtained by randSRRQR-$\tau$ (Algorithm \ref{alg:RSRRQR_k}) for a tolerance of $\tau = 10^{-10}$, SRRQR for a tolerance of $\tau = 10^{-10}$, and QRCP are very close to each other, as it can be seen in Figure \ref{fig:Devil1}. The results obtained by QRCP and SRRQR are equal, since SRRQR does no column interchange in the inner loop of each iteration. We observe that the last ratio is slightly larger for QRCP than for randSRRQR-$\tau$. The last ratios are shown in Table \ref{tab:Devil1}. The singular values of $\bM$ and those of $\bR_{11}$ computed by randSRRQR-$\tau$, SRRQR, and QRCP are displayed in  Figure \ref{fig:Devil2}, along with the corresponding L-values and R-values. It can be seen that the L-values approximate the singular values of $\bM$ better than the R-values or the singular values of $\bR_{11}$ computed by the three methods. 

% plots for Devil's stairs
\begin{figure}[htbp]
\begin{minipage}[t]{0.48\textwidth}
\centering
    \includegraphics[width = 8.23cm]{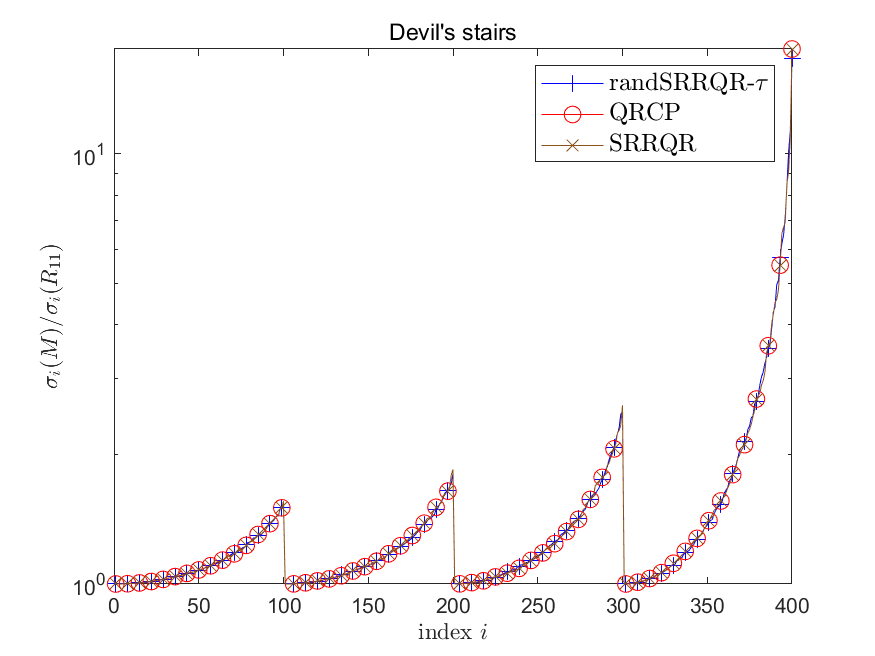}
    \caption{Ratios $\ms_i(\bM)/\ms_i(\bR_{11})$ of $8192\times500$ Devil's stairs}
    \label{fig:Devil1}
\end{minipage}
    \begin{minipage}[t]{0.48\textwidth}
    \centering
    \includegraphics[width = 8.23cm]{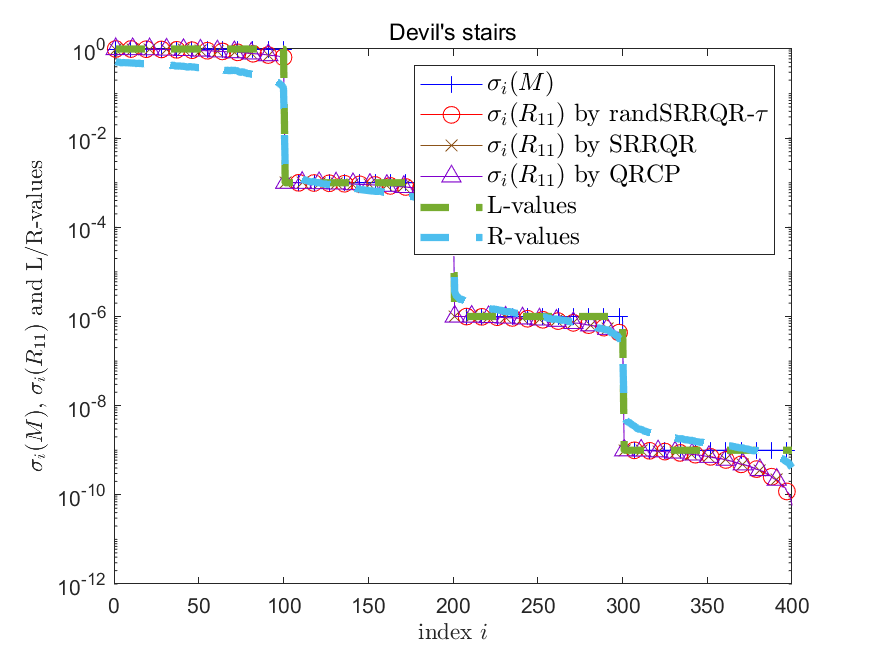}
    \caption{Singular values and L/R-values for $8192\times500$ Devil's stairs}
    \label{fig:Devil2}
\end{minipage}
\end{figure}

%For devil's stairs, although QRCP works on this matrix, our algorithm can still show some of its advantages. We keep the other parameters same and let $k$, the dimension of $\bR_{11}$, change. In the table \ref{tab:devil1}, we show that the maximum ratio computed by our algorithm is always slightly smaller than those computed by QRCP.

%When $k$ is set as $60$, we plot all the ratios as is shown in the figure\ref{fig:devil2}. We notice that $\frac{\ms_i(M)}{\ms_i{R_{11}}}$ is relatively larger when there is a gap between $\ms_i(M)$ and $\ms_{i+1}(M)$. Compared to randomized SVD, most of our approximated singular values have the same precision but the last few ones have some deviation. 
%We also test the truncated QLP approximation for devil's stairs. We set $\ell$ from 20 to 45 and get the corresponding approximation $\hat{\bM}_{\ell}$ with rank $\ell$. In the figure \ref{fig:devil3}, we plot the L-values computed by QRCP and Rand SRRQR. The L-values computed by Rand SRRQR have some oscillation. In the figure \ref{fig:devil4}, we plot the approximation errors $\|\bM - \hat{\bM}_{\ell} \|_2$, where the red dashed line we add is the theoretical bounds shown in the theorem \ref{th_qdp}  and the yellow one is the errors from the truncated SVD. Because of the gaps between singular values of $\bM$, our approximation is very close to the optimal approximation.   

% plots for Stewart matrix
\begin{figure}[htbp]
\begin{minipage}[t]{0.48\textwidth}
\centering
    \includegraphics[width = 8.23cm]{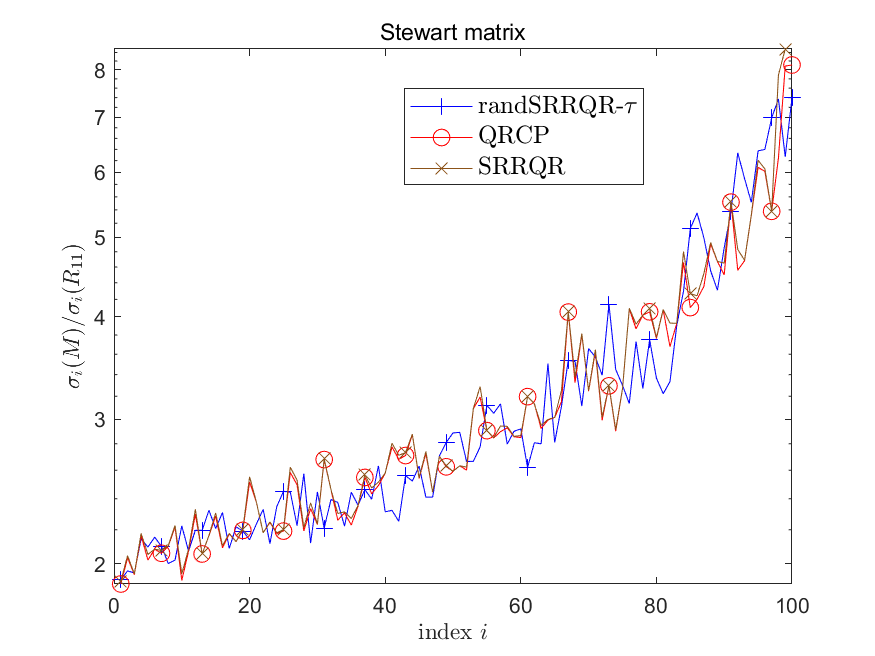}
    \caption{Ratios $\ms_i(\bM)/\ms_i(\bR_{11})$ of $8192\times500$ Stewart matrix}
    \label{fig:Stewart1}
\end{minipage}
    \begin{minipage}[t]{0.48\textwidth}
    \centering
    \includegraphics[width = 8.23cm]{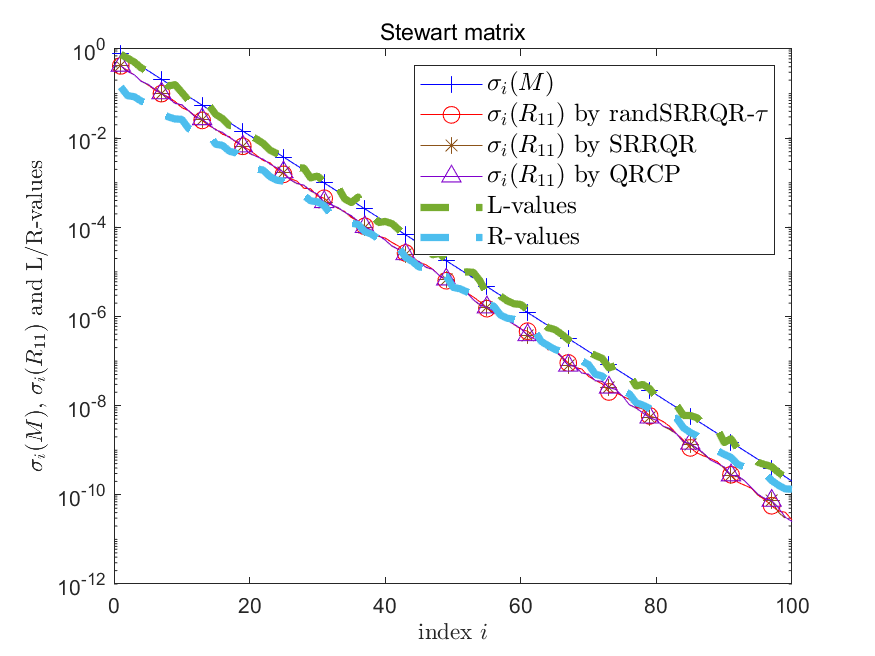}
    \caption{Singular values and L/R-values for $8192\times500$ Stewart matrix}
    \label{fig:Stewart2}
\end{minipage}
\end{figure}

% plots for H-C matrix
 \begin{figure}[htbp]
\begin{minipage}[t]{0.48\textwidth}
\centering
    \includegraphics[width = 8.23cm]{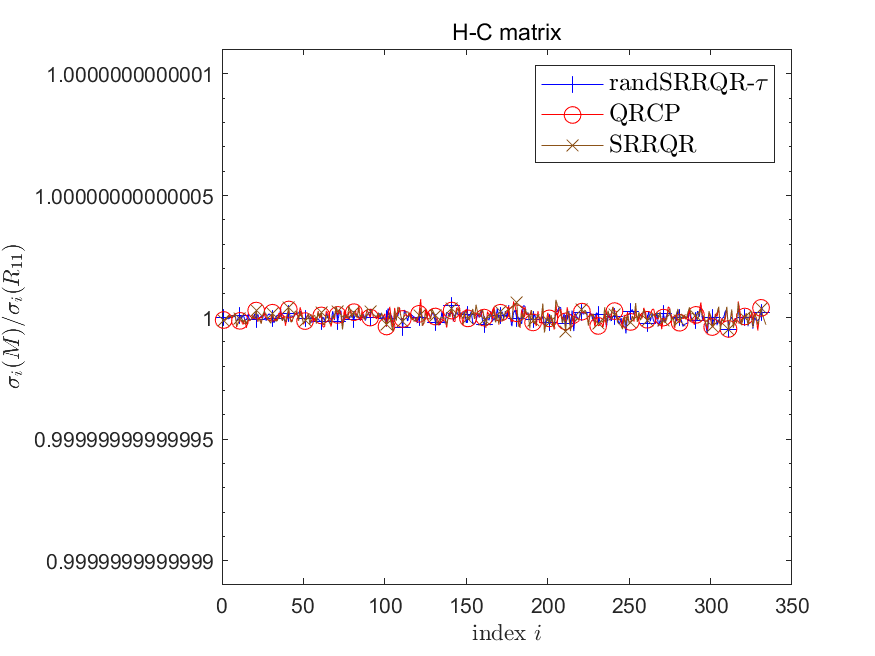}
    \caption{Ratios $\ms_i(\bM)/\ms_i(\bR_{11})$ of $8192\times500$ H-C matrix}
    \label{fig:hc1}
\end{minipage}
    \begin{minipage}[t]{0.48\textwidth}
    \centering
    \includegraphics[width = 8.23cm]{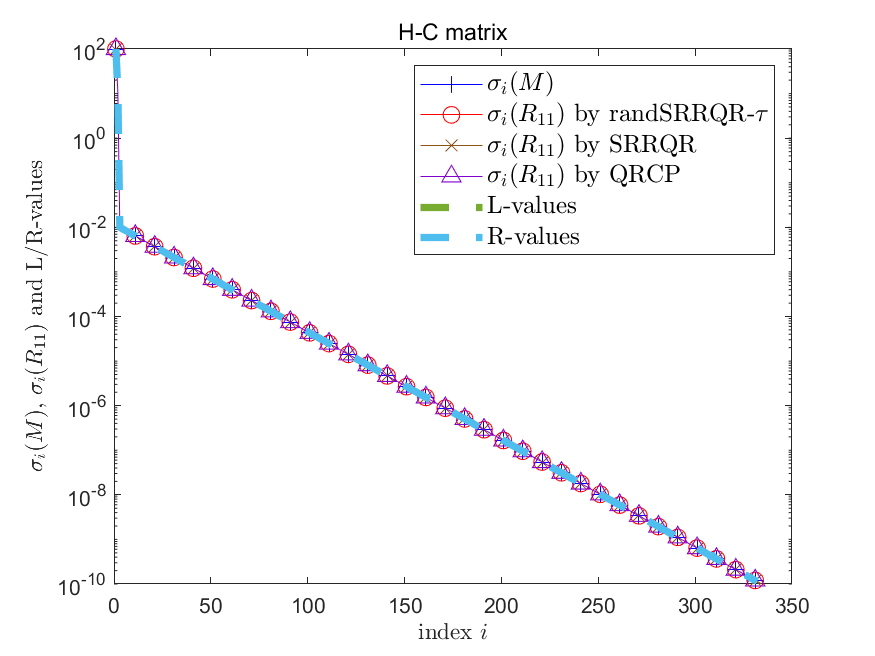}
    \caption{Singular values and L/R-values for $8192\times500$ H-C matrix}
    \label{fig_hc2}
\end{minipage}
\end{figure}

We now consider Stewart's matrix\cite{qlp}, which is a perturbation of a matrix with prescribed singular values $1,\ q,\ q^2,\ \cdots,\ q^{n/2},\ 0,\ 0,\ \cdots,\ 0$. The perturbation is obtained by adding a random matrix $c\cdot$rand(m,n), with $c$ being a small constant. In our tests, we set $q$ as $0.8$ and $c$ as $q^{n/2}$. We set $\tau$ as $10^{-10}$ for randSRRQR-$\tau$ and  SRRQR. The results obtained for a matrix of dimension $8192\times500$ are presented in Figures \ref{fig:Stewart1} and \ref{fig:Stewart2}. We can see that the ratios $\ms_i(\bM)/\ms_i(\bR_{11})$ computed by three methods, randSRRQR-$\tau$, QRCP, and SRRQR, are comparable to each other. We can see that the approximations of singular values computed by randSRRQR-$\tau$ (Algorithm \ref{alg:RSRRQR_k}) are comparable to those computed by QRCP. The results computed by QRCP and SRRQR are almost the same. The numerical rank determined by SRRQR is $99$ while the one determined by randSRRQR-$\tau$ is $100$. As shown in Figure \ref{fig:Stewart2}, the L-values approximate the singular values of $\bM$ better than the R-values or the singular values of $\bR_{11}$ computed by the three methods.
%When $\ell$ is from 40 to 60, we compute the low-rank approximations $\hat{\bM}_{\ell}$ and compute the $\ell_2$-norm of errors in the figure \ref{fig:stewart3}. The two dashed line is the same as we explained before. We can see that $\hat{\bM}_{\ell}$ approximates $\bM$ very well when $\ell<47$. 

\begin{comment}
    \begin{figure}[htbp]
    \centering
    \includegraphics[width=0.6\textwidth]{stewart_4_80k.png}
    \caption{$\ell_2$-norm of approximation errors}
    \label{fig:stewart3}
\end{figure}
\end{comment}

Our last test case is the H-C matrix, defined slightly different from \cite{qlp_lra}. An H-C matrix of dimension $m \times n$ is defined as $\bM = \bU \pmb{\Sigma}$, where $\bU\in\Rmn$ is a random orthogonal matrix and $\pmb{\Sigma}\in\Rnn$ is a diagonal matrix whose diagonal elements are $100$, $10$ and $n-2$ numbers logarithmically evenly spaced between $10^{-2}$ and $10^{-14}$.  For a H-C matrix with dimension $8192 \times 500$, the results are displayed in Figures \ref{fig:hc1} and \ref{fig_hc2}. We set $\tau$ as $10^{-10}$ for randSRRQR-$\tau$ and  SRRQR. We can see that these ratios are very close to $1$, thus randSRRQR-$\tau$, SRRQR and QRCP are all effective in approximating the singular values of this matrix. In Figure \ref{fig_hc2}, we see that $\ms_i(\bR_{11})$ computed by the three methods and the L/R-values approximate $\ms_i(\bM)$ very well.
%When $\ell$ is from 40 to 60, we compute the low-rank approximations $\hat{\bM}_{\ell}$ and compute the $\ell_2$-norm of errors in the figure \ref{fig:hc3}. The two dashed line is the same as we explained before. We can see that $\hat{\bM}_{\ell}$ approximates $\bM$ very well when $\ell<47$. 

\begin{comment}
    \begin{figure}[htbp]
    \centering
    \includegraphics[width=0.6\textwidth]{hc_4.png}
    \caption{$\ell_2$-norm of approximation errors}
    \label{fig:hc3}
\end{figure}
\end{comment}
\section{Conclusions}
This paper discusses a randomized algorithm that can be used to effectively select columns from a matrix while revealing its spectrum. It thus leverages random sketching to reduce the computational and communication costs, enabling effective low-rank approximation, rank estimation, singular value approximation, and null-space approximation. A key contribution of this paper is the analysis of properties that can be preserved through sketching as the ratios of volumes of matrices for which only one column is different. While this property is used to prove strong rank-revealing properties of the algorithms discussed in this paper, it can be used in other contexts related to finding the submatrix of maximal volume of a matrix\cite{maximal_volume_Goreinov}. Future work includes exploring the usage of mixed precision and the application of the randomized strong RRQR factorization in other low-rank matrix approximation algorithms as interpolative decompositions.

\section{Acknowledgements}
This project has received funding from the European Research
Council (ERC) under the European Union’s Horizon 2020 research and innovation program (grant agreement No 810367).

 \bibliographystyle{plain}
\bibliography{bounds}
\end{document}